\documentclass[11pt]{amsart}
\pdfoutput=1
\usepackage{amssymb}
\usepackage{stmaryrd}
\usepackage{microtype}
\usepackage{mathtools}

\usepackage{epigraph}

\usepackage[numbers]{natbib}

\usepackage[all,tips]{xy}
\SelectTips{cm}{11}

\usepackage{aliascnt}
\usepackage{hyperref}
\hypersetup{
  pdftitle={Symplectic groupoids and discrete constrained Lagrangian mechanics},
  pdfauthor={J. C. Marrero, D. Mart\'in de Diego, and A. Stern},
  pdfsubject={MSC 2010: Primary 70G45; Secondary 53D17, 37M15.},
  bookmarksopen=false
}

\makeatletter
\renewcommand{\subsection}{\@startsection{subsection}{2}%
  \z@{.5\linespacing\@plus.7\linespacing}{-.5em}%
  {\normalfont\bfseries\boldmath}}
\makeatother

        \theoremstyle{plain} %--default
        \newtheorem{theorem}             {Theorem}  [section]
        \newaliascnt{lemma}{theorem}
        \newtheorem{lemma} [lemma]{Lemma}
        \aliascntresetthe{lemma}
        \newaliascnt{corollary}{theorem}
        \newtheorem{corollary}  [corollary]{Corollary}
        \aliascntresetthe{corollary}

        \theoremstyle{definition}
        \newaliascnt{definition}{theorem}
        \newtheorem{definition} [definition]{Definition}
        \aliascntresetthe{definition}
        
        \newaliascnt{example}{theorem}
        \newtheorem{example} [example]{Example}
        \aliascntresetthe{example}

        \theoremstyle{remark}
        \newtheorem{remark}              [theorem]{Remark}

        \newtheorem*{acknowledgment}      {Acknowledgments}

\begin{document}
\title[Symplectic groupoids and discrete mechanics]{Symplectic
  groupoids and discrete constrained Lagrangian mechanics}

\author[J.~C.~Marrero]{Juan Carlos Marrero}
\address{Unidad Asociada
  ULL--CSIC ``Geometr{\'\i}a Diferencial y
  Mec\'anica Geom\'etrica''\\
  Departamento de Matem\'atica Fundamental\\
  Facultad {de} Matem\'aticas\\
  Universidad de La Laguna\\
  La Laguna, Tenerife, Canary Islands, Spain}
\email{jcmarrer@ull.es}

\author[D.~Mart\'in de Diego]{David Mart\'in de Diego}
\address{Instituto de Ciencias Matem\'aticas (CSIC--UAM--UC3M--UCM)\\
Campus de Cantoblanco, UAM\\
C/ Nicolas Cabrera 15\\ 28049 Madrid, Spain}
\email{david.martin@icmat.es}

\author[A.~Stern]{Ari Stern}
\address{Department of Mathematics\\
Washington University in St.~Louis\\
Campus Box 1146\\
One Brookings Drive\\
St.~Louis, Missouri 63130-4899, USA}
\email{astern@math.wustl.edu}

\begin{abstract}
  In this article, we generalize the theory of discrete Lagrangian
  mechanics and variational integrators in two principal directions.
  First, we show that Lagrangian submanifolds of symplectic groupoids
  give rise to discrete dynamical systems, and we study the properties
  of these systems, including their regularity and reversibility, from
  the perspective of symplectic and Poisson geometry.  Next, we use
  this framework---along with a generalized notion of generating
  function due to \citeauthor{SnTu1972}---to develop a theory of
  \emph{discrete constrained Lagrangian mechanics}.  This allows for
  systems with arbitrary constraints, including those which are
  non-integrable (in an appropriate discrete, variational sense).  In
  addition to characterizing the dynamics of these constrained
  systems, we also develop a theory of reduction and Noether
  symmetries, and study the relationship between the dynamics and
  variational principles.  Finally, we apply this theory to discretize
  several concrete examples of constrained systems in mechanics and
  optimal control.
\end{abstract}

\date{January 1, 2014}


\maketitle

\section{Introduction}

\setlength{\epigraphwidth}{.6\textwidth}
\epigraph{Understand[ing] in depth the symplectic geometry of the
  generating function method \ldots should lead to a method for
  nonholonomic constraints.}{McLachlan and Scovel, \emph{A survey of
    open problems in symplectic integration} \citep{McSc1996}}

In this article, we generalize the theory of discrete Lagrangian
mechanics and variational integrators in two principal directions.
First, we show that Lagrangian submanifolds of symplectic groupoids
give rise to discrete dynamical systems, and we study the properties
of these systems, including their regularity and reversibility, from
the perspective of symplectic and Poisson geometry.  Next, we use this
framework---along with a generalized notion of generating function due
to \citet{SnTu1972}---to develop a theory of \emph{discrete
  constrained Lagrangian mechanics}.  This allows for systems with
arbitrary constraints, including those which are non-integrable (in an
appropriate discrete, variational sense).  In addition to
characterizing the dynamics of these constrained systems, we also
develop a theory of reduction and Noether symmetries, and study the
relationship between the dynamics and variational principles.
Finally, we apply this theory to discretize several concrete examples
of constrained systems in mechanics and optimal control.

Before giving a more in-depth overview of the paper, and summarizing
the key results, let us first give a brief background on discrete
Lagrangian mechanics, in order to provide the context for the present
work.

\subsection{Background}
The subject of \emph{discrete Lagrangian mechanics} concerns the study
of certain discrete dynamical systems on manifolds.  As the name
suggests, these discrete systems exhibit many geometric features which
are analogous to those in continuous Lagrangian mechanics: in
particular, the dynamics of these systems satisfy variational
principles, have symplectic or Poisson flow maps, conserve momentum
maps associated to Noether-type symmetries, and admit a theory of
reduction.  While discrete Lagrangian systems are quite mathematically
interesting, in their own right, they also have important applications
to structure-preserving numerical simulation of dynamical systems in
geometric mechanics and optimal control theory.

In the simplest form of discrete Lagrangian mechanics, one begins with
a function $ L \colon Q \times Q \rightarrow \mathbb{R} $, called the
\emph{discrete Lagrangian}, where $Q$ is some configuration manifold.
A trajectory $ q _0, \ldots, q _n \in Q $ is a solution of the system
if it extremizes the \emph{discrete action sum},
\begin{equation*}
  S \left( q _0, \ldots, q _n \right) = \sum _{ k = 1 } ^n L \left( q
    _{k-1} , q _k \right) ,
\end{equation*} 
where the endpoints $ q _0 $ and $ q _n $ are held fixed.  This is
essentially a discrete version of Hamilton's principle of stationary
action, where the tangent bundle $ T Q $ has been replaced by $ Q
\times Q $, and the action integral has been replaced by an action
sum.  Solving this variational principle, one obtains the
\emph{discrete Euler--Lagrange equations},
\begin{equation*}
  0 = \frac{ \partial }{ \partial q _k } S \left( q _0, \ldots, q _n
  \right) =  \partial _0 L \left( q _k , q _{ k + 1 } \right) + \partial _1 L
  \left( q _{ k - 1 } , q _k \right) ,
\end{equation*}
for $ k = 1 , \ldots, n - 1$, where $ \partial _0 L $ and $ \partial
_1 L $ denote the partial derivatives of $L$ with respect to the first
and second arguments, respectively.  These discrete Euler--Lagrange
equations define an implicit function (i.e., a relation), mapping $
\left( q _{ k - 1 } , q _k \right) \mapsto \left( q _k , q _{ k + 1 }
\right) $. If $L$ is sufficiently nondegenerate, then this relation is
the graph of a flow map $ Q \times Q \rightarrow Q \times Q $, at
least locally.

Alternatively, but equivalently, one can define a pair of
\emph{discrete Legendre transformations} $ \mathbb{F} ^{ \pm } L
\colon Q \times Q \rightarrow T ^\ast Q $, given by
\begin{equation*}
  \mathbb{F}  ^- L \left( q _{ k - 1 }, q _k \right) = - \partial _0 L
  \left( q _{ k - 1 }, q _k \right), \qquad   \mathbb{F}  ^+ L \left( q
    _{ k - 1 }, q _k \right) = \partial _1 L \left( q _{ k - 1 }, q _k
  \right).
\end{equation*} 
This defines a symplectic relation on $ T ^\ast Q $, given by
\begin{equation*}
  p _{k-1} =   \mathbb{F}  ^- L \left( q _{ k - 1 }, q _k \right),
  \qquad p _k = \mathbb{F}  ^+ L \left( q _{ k - 1 }, q _k \right),
\end{equation*} 
so if the discrete Legendre transforms are (locally) invertible, then
this gives a (local) symplectic automorphism
\begin{equation*}
  \mathbb{F}  ^+ L \circ \left( \mathbb{F}  ^- L \right) ^{-1} \colon
  T ^\ast Q \rightarrow T^\ast Q , \quad \left( q _{ k -1 } , p _{ k
      -1 } \right) \mapsto \left( q _k , p _k \right) ,
\end{equation*} 
which can be interpreted as the discrete Hamiltonian flow of the
system.  More precisely, if $ \omega $ denotes the canonical
symplectic form on $ T^\ast Q $, then $ L \colon Q \times Q
\rightarrow \mathbb{R} $ is a generating function for the Lagrangian
submanifold $ \mathrm{d} L \left( Q \times Q \right) \subset \left( T
  ^\ast Q, - \omega \right) \times \left( T ^\ast Q , \omega \right)
$, which can be identified with the graph of the discrete Hamiltonian
flow in the nondegenerate case.  Observe that 
\begin{equation*}
  \mathbb{F} ^+ L \left(
    q _{ k - 1 }, q _k \right) = p _k = \mathbb{F} ^- L \left( q _k , q
    _{ k + 1 } \right) \Leftrightarrow \partial _0 L \left( q _k , q
    _{ k + 1 } \right) + \partial _1 L  \left( q _{ k - 1 } , q _k
  \right) = 0,
\end{equation*} 
so this approach is equivalent to the previous one.

Numerical methods which are constructed in this fashion are called
\emph{variational integrators}, due to the key role played by the
variational principle. This approach to discretizing Lagrangian
systems was put forward in seminal papers by \citet{Suris1990},
\citet{MoVe1991}, and others in the early 1990s, and the general
theory was developed over the subsequent decade (see \citet{MaWe2001}
for a comprehensive overview).

\citet{Weinstein1996} observed that these systems could be understood
as a special case of a more general theory, describing discrete
Lagrangian mechanics on arbitrary Lie groupoids.  Given a Lie groupoid
$ G \rightrightarrows Q $ and a discrete Lagrangian $ L \colon G
\rightarrow \mathbb{R} $, a composable sequence of elements $ g _1,
\ldots, g _n \in G $ is a solution trajectory if it extremizes the
discrete action sum 
\begin{equation*}
  S \left( g _1, \ldots, g _n \right) = \sum _{ k = 1 } ^n L \left( g
    _k \right) ,
\end{equation*} 
where the product $ g _1 \cdots g _n = g $ is held fixed.  This
variational principle yields a discrete Lagrangian flow map $ g _k
\mapsto g _{ k + 1 } $.  Equivalently, $L$ can be interpreted as a
generating function for the Lagrangian submanifold $ \mathrm{d} L (G)
\subset T ^\ast G $.  Since $ T ^\ast G $ is a symplectic groupoid
over the dual Lie algebroid $ A ^\ast G $, which has a canonical
Poisson structure, this Lagrangian submanifold defines a Poisson
relation on $ A ^\ast G $.  In the special case of the so-called pair
groupoid $ Q \times Q \rightrightarrows Q $, this recovers the
previous theory.  We will recall the details of this approach in
\autoref{sec:groupoids}, along with some of the more recent
developments that appeared in \citet{MaMaMa2006}, including an
explicit characterization of the discrete Lagrangian and Hamiltonian
dynamics.

\subsection{Organization of the paper}
In \autoref{sec:groupoids}, we begin by recalling the basic
definitions and properties of Lie groupoids and algebroids,
highlighting the close relationship between symplectic groupoids and
Poisson manifolds.  After laying these foundations, we next describe
how Lagrangian submanifolds of symplectic groupoids give rise to
discrete dynamics, and how this can be seen as an abstract
generalization of discrete Lagrangian mechanics.  The main result of
this section, \autoref{thm:lagrangianBisection}, is a significant
generalization of earlier results characterizing the regularity and
reversibility of discrete Lagrangian systems.  We prove, in a
generalized sense, that if either of the discrete Legendre transforms
is a (local) diffeomorphism, then both are, and hence the discrete
flow is a (local) Poisson automorphism.

Next, in \autoref{sec:constraints}, we develop a theory of discrete
constrained Lagrangian mechanics, for systems which are restricted to
some constraint submanifold of a Lie groupoid.  In this case, rather
than having a discrete Lagrangian defined on the entire groupoid, we
only have $ L \colon N \rightarrow \mathbb{R} $ for some $ N \subset G
$.  When $N$ is a proper submanifold, this cannot be a generating
function in the usual sense; however, we show that a more general
notion of generating function, due to \citet{SnTu1972}, can be used to
generate a Lagrangian submanifold $ \Sigma _L \subset T ^\ast G $.
Using this approach, we characterize the regularity of discrete
constrained Lagrangian systems, and obtain explicit formulas for the
discrete Legendre transforms and discrete Euler--Lagrange equations.
Finally, after showing that $ \Sigma _L \subset T ^\ast G $ has the
structure of an affine bundle over $N$ associated to the conormal
bundle $ \nu ^\ast N $, we discuss how the dynamics on $ \Sigma _L $
can thus be interpreted as the evolution of configurations (on the
base) and Lagrange multipliers (on the fibers).

In \autoref{sec:morphisms}, we introduce morphisms of discrete
constrained Lagrangian systems, and use this to develop the reduction
theory for these systems.  The notion of Noether symmetries is also
introduced, and we prove the appropriate discrete version of Noether's
theorem, relating symmetries to conserved quantities.

Subsequently, in \autoref{sec:variational}, we show how the dynamics
of discrete constrained Lagrangian systems can be derived from a
constrained variational principle on $N$, and demonstrate that these
dynamics are equivalent to those obtained previously via the
generating function approach.

After this, in \autoref{sec:examples}, we examine examples of discrete
constrained Lagrangian systems for several applications.  These
applications include constrained mechanics and optimal control on Lie
groups, an optimal control problem with non-integrable constraints
known as the plate-ball system, and time-dependent constrained
mechanics with either fixed or adaptive time-stepping.

Throughout the paper, we have occasion to draw on certain technical
results involving bisections of Lie groupoids, which are also
instrumental to the multiplicative structure of the cotangent groupoid
and to the proof of \autoref{thm:variationalIsomorphism}.  We provide
some supplementary details and discussion of these results in
\autoref{app:bisections}.

\section{Symplectic groupoids and discrete Poisson dynamics}
\label{sec:groupoids}

\subsection{Lie groupoids and algebroids} Let us first briefly recall
some definitions pertaining to Lie groupoids and Lie algebroids.  This
will provide an abbreviated reference and fix the notation used
throughout the remainder of the paper.  For a detailed treatment of
this rich topic, see the comprehensive work of \citet{Mackenzie2005}.

\begin{definition}
  A \emph{groupoid} is a small category in which every morphism is
  invertible.  That is, the groupoid denoted $ G \rightrightarrows Q $
  consists of a set of objects $Q$, a set of morphisms $G$, and the
  following structural maps:
  \begin{enumerate}
  \item a \emph{source map} $ \alpha \colon G \rightarrow Q $ and
    \emph{target map} $ \beta \colon G \rightarrow Q $;
  \item a \emph{multiplication map} $ m \colon G _2 \rightarrow G $, $
    \left( g, h \right) \mapsto gh $, where
    \begin{equation*}
      G _2 = G \mathrel{_\beta \times _\alpha } G = \left\{ \left( g,
          h \right) \in G \times G \;\middle\vert\; \beta (g) = \alpha
        (h) \right\} 
    \end{equation*}
    is called the set of \emph{composable pairs}, such that
    multiplication is associative whenever defined;
  \item an \emph{identity section} $ \epsilon \colon Q \rightarrow G $
    of $\alpha$ and $\beta$, such that for all $ g \in G $,
    \begin{equation*}
      \epsilon \left( \alpha (g) \right) g = g = g \epsilon \left(
        \beta (g) \right) ;
    \end{equation*}
  \item and an \emph{inversion map} $ i \colon G \rightarrow G $, $ g
    \mapsto g ^{-1} $, such that for all $ g \in G $,
    \begin{equation*}
      g g ^{-1} = \epsilon \left( \alpha (g) \right) , \qquad g ^{-1} g
      = \epsilon \left( \beta (g) \right) .
    \end{equation*}
  \end{enumerate}
\end{definition}

\begin{remark}
  Intuitively, groupoids can be viewed as either ``categories with
  extra structure'' or as ``groups with missing structure,'' and both
  perspectives are useful.  For instance, it can be helpful to see $ g
  \in G $ as an ``arrow'' from $ \alpha (g) $ to $ \beta (g) $, where
  the multiplicative structure defines the composition of arrows $
  \xymatrix{\bullet \ar@/^/[r]_g \ar@/^4ex/[rr]_{ g h } & \bullet
    \ar@/^/[r] _h & \bullet} $.  Alternatively, the groupoid $G$ can
  be thought of as being a weaker version of a group, where
  multiplication is defined only partially (on $ G _2 \subset G \times
  G $) rather than totally (on $ G \times G $).
\end{remark}

Next, for any element $ g \in G $ of a groupoid, there are associated
left and right translation maps, which act, respectively, on those
elements with which $g$ is composable on the left and on the right.

\begin{definition}
  Given a groupoid $ G \rightrightarrows Q $ and an element $ g \in G
  $, define the \emph{left translation} $ \ell _g \colon \alpha ^{-1}
  \left( \beta (g) \right) \rightarrow \alpha ^{-1} \left( \alpha (g)
  \right) $ and \emph{right translation} $ r _g \colon \beta ^{-1}
  \left( \alpha (g) \right) \rightarrow \beta ^{-1} \left( \beta (g)
  \right) $ by $g$ to be
  \begin{equation*}
    \ell _g (h) = g h , \qquad r _g (h) = h g .
  \end{equation*}
\end{definition}

For discrete mechanics, we will focus on a particular class of
groupoids, Lie groupoids, which---much like Lie groups, of which they
are a generalization---have a differential structure in addition to
(and compatible with) their algebraic structure.

\begin{definition}
  A \emph{Lie groupoid} is a groupoid $ G \rightrightarrows Q $ where
  $G$ and $Q$ are differentiable manifolds, $\alpha$ and $\beta$ are
  submersions, and the multiplication map $ m $ is differentiable.
\end{definition}

\begin{remark}
  Observe that it is necessary for $\alpha$ and $\beta$ to be
  submersions so that $ G _2 $ is a differentiable manifold;
  otherwise, one could not make sense of $m$ as a differentiable map.
  It also follows from the definition that $m$ is a submersion,
  $\epsilon$ is an immersion, and $i$ is a diffeomorphism.
\end{remark}

In a natural way, one can now introduce the notion of left- and
right-invariant vector fields on a Lie groupoid.

\begin{definition}
  Given a Lie groupoid $ G \rightrightarrows Q $, a vector field $ \xi
  \in \mathfrak{X} (G) $ is \emph{left-invariant} if it is
  $\alpha$-vertical (i.e., $ T \alpha (\xi) = 0 $) and $ T _h \ell _g
  \left( \xi (h) \right) = \xi \left( g h \right) $ for all $ \left(
    g, h \right) \in G _2 $.  Similarly, $\xi$ is
  \emph{right-invariant} if it is $\beta$-vertical (i.e., $ T \beta
  (\xi) = 0 $), and $ T _h r _g \left( \xi (h) \right) = \xi \left( h
    g \right) $ for all $ \left( h, g \right) \in G _2 $.
\end{definition}

Just as the ``infinitesimal version'' of a Lie group is a Lie algebra,
the infinitesimal version of a Lie groupoid is a Lie algebroid.  We
will first define Lie algebroids as abstract structures, in their own
right, before subsequently recalling the definition of the Lie
algebroid associated to a particular Lie groupoid.

\begin{definition}
  A \emph{Lie algebroid} is a real vector bundle $ A \rightarrow Q $,
  equipped with a Lie bracket $ \left\llbracket \cdot , \cdot
  \right\rrbracket $ on its space of sections $ \Gamma (A) $, and with
  a bundle map $ \rho \colon A \rightarrow T Q $ called the
  \emph{anchor map}.  Furthermore, if we also denote by $ \rho \colon
  \Gamma (A) \rightarrow \mathfrak{X} (Q) $ the homomorphism of $ C
  ^\infty (Q) $-modules of sections induced by the anchor, then we
  require that this satisfies the ``Leibniz rule,''
  \begin{equation*}
    \left\llbracket X, f Y \right\rrbracket = f\left\llbracket X,
      Y \right\rrbracket + \rho (X) \left[ f \right] Y ,
  \end{equation*}
  for all $ X, Y \in \Gamma (A) $ and $ f \in C ^\infty (Q) $.
\end{definition}

\begin{remark}
  Using the Jacobi identity, it is not hard to show that the induced
  map $ \rho \colon \Gamma (A) \rightarrow \mathfrak{X} (Q) $ is a Lie
  algebra homomorphism, where $ \mathfrak{X} (Q) $ is endowed with the
  usual Jacobi--Lie bracket $ \left[ \cdot, \cdot \right] $ for vector
  fields.
\end{remark}

\begin{definition}
  Given a Lie groupoid $ G \rightrightarrows Q $, the \emph{associated
    Lie algebroid} $ A G \rightarrow Q $ is defined by its fibers $ A
  _q G = \ker T _{ \epsilon (q) } \alpha $, i.e., the space of
  $\alpha$-vertical tangent vectors at the identity section, for each
  $ q \in Q $.  There is a bijection between sections $ X \in \Gamma
  \left( A G \right) $ and left-invariant vector fields $
  \overleftarrow{ X } \in \mathfrak{X} (G) $, defined by
  \begin{equation}\label{eqn:leftInvariant}
    \overleftarrow{ X } (g) =  T _{ \epsilon \left( \beta (g)
        \right) } \ell _g \left( X \left( \beta (g) \right) \right) .
  \end{equation}
  The Lie algebroid structure on $ A G $ is then given by the bracket
  $ \left\llbracket \cdot , \cdot \right\rrbracket $ and anchor $\rho$
  satisfying the conditions
  \begin{equation*}
    \overleftarrow{ \left\llbracket X , Y \right\rrbracket } = [
    \overleftarrow{ X } , \overleftarrow{ Y } ], \qquad \rho (X) (q) =
    T _{ \epsilon (q) } \beta \left( X (q) \right) ,
  \end{equation*}
  for all $ X, Y \in \Gamma \left( A G \right) $ and $ q \in Q $.
\end{definition}

\begin{remark}
  Alternatively, one can also establish a bijection between sections $
  X \in \Gamma \left( A G \right) $ and right-invariant vector fields
  $ \overrightarrow{ X } \in \mathfrak{X} (Q) $, defined by
  \begin{equation}\label{eqn:rightInvariant}
    \overrightarrow{ X } (g) = - T _{ \epsilon \left( \alpha (g) \right)
    } \left( r _g \circ i \right) \left( X \left( \alpha (g) \right)
    \right) ,
  \end{equation}
  which yields the Lie bracket relation
  \begin{equation*}
    \overrightarrow{ \left\llbracket X, Y \right\rrbracket } = - [
    \overrightarrow{ X } , \overrightarrow{ Y } ] .
  \end{equation*}
  Thus, $ X \mapsto \overleftarrow{ X } $ is a Lie algebra
  isomorphism, while $ X \mapsto \overrightarrow{ X } $ is a Lie
  algebra anti-isomorphism.

  Note that for every $ v \in A _q G $, we have the following
  expression for the tangent to the inversion map:
  \begin{equation}
    \label{eqn:tangentInversion}
    T _{ \epsilon (q) } i (v) = - v + T _{ \epsilon (q) } \left(
      \epsilon \circ \beta \right) (v) .
  \end{equation}
  This can be seen by taking $ v = \dot{ g } (0) $ for some
  $\alpha$-vertical path $ g (t) $, and observing that $ T _{ \left(
      \epsilon (q), \epsilon (q) \right) } m \left( 0, v \right) =
  \frac{\mathrm{d}}{\mathrm{d}t} m \left( \epsilon (q), g (t) \right)
  \rvert _{ t = 0 } = \dot{ g } (0) = v $; likewise, for the
  $\beta$-vertical path $ g ^{-1} (t) $, we have $ T _{ \left(
      \epsilon (q) , \epsilon (q) \right) } m \left( T _{ \epsilon (q)
    } i (v), 0 \right) = T _{ \epsilon (q) } i (v) $.  Therefore, $ T
  _{ \left( \epsilon (q), \epsilon (q) \right) } m \left( T _{
      \epsilon (q) } i (v), v \right) = v + T _{ \epsilon (q) } i (v)
  $---but since $ m \left( g ^{-1} (t) , g (t) \right) = \epsilon
  \left( \beta (g) \right) $, we obtain $ v + T _{ \epsilon (q) } i
  (v) = T _{ \epsilon (q) } \left( \epsilon \circ \beta \right) (v) $,
  which rearranges to give the desired equality.
\end{remark}

Finally, we discuss a subclass of Lie groupoids having even more
additional structure.  These are the symplectic groupoids, which are
endowed with a symplectic manifold structure that is ``compatible''
with the Lie groupoid structure, in a sense which we will now define
precisely.

\begin{definition}
  A \emph{symplectic groupoid} is a Lie groupoid $ \widetilde{ G }
  \rightrightarrows P $, such that
  \begin{enumerate}
  \item $ ( \widetilde{ G } , \widetilde{ \omega }
    ) $ is a symplectic manifold,
  \item the graph of $ \widetilde{ m } \colon \widetilde{ G } _2
    \rightarrow \widetilde{ G } $ is a Lagrangian submanifold of $
    \widetilde{ G } ^- \times \widetilde{ G } ^- \times \widetilde{ G
    } $, where $ \widetilde{ G } ^- = ( \widetilde{ G }, - \widetilde{
      \omega } ) $ denotes the negative symplectic structure.
  \end{enumerate}
\end{definition}

If $ \widetilde{ G } \rightrightarrows P $ is a symplectic groupoid
and $ \widetilde{ \omega } $ is the symplectic $2$-form on $
\widetilde{ G } $, then one may prove that $ (\ker T _\mu \widetilde{
  \alpha }) ^{ \widetilde{ \omega } } = \ker T _\mu \widetilde{ \beta
} $ for $ \mu \in \widetilde{ G } $, where $ (\ker T _\mu \widetilde{
  \alpha }) ^{ \widetilde{ \omega } } $ denotes the symplectic
orthogonal complement of the subspace $ \ker T _\mu \widetilde{ \alpha
} $.  That is, the subspaces of $ \widetilde{ \alpha } $-vertical and
$ \widetilde{ \beta } $-vertical tangent vectors are symplectic
orthogonal to one another.  Moreover, there exists a unique Poisson
structure on $P$ such that $ \widetilde{ \alpha } \colon \widetilde{ G
} \rightarrow P $ is anti-Poisson and $ \widetilde{ \beta } \colon
\widetilde{ G } \rightarrow P $ is Poisson.  In addition, the identity
section $ \widetilde{ \epsilon } \colon P \rightarrow \widetilde{ G }
$ is a Lagrangian immersion. (See
\citet{CoDaWe1987,Mackenzie2005,Marle2005}.)

One example of a symplectic groupoid, in particular, lies at the heart
of discrete Lagrangian mechanics: this is the cotangent groupoid.
Given a Lie groupoid $ G \rightrightarrows Q $, let $ A ^\ast G
\rightarrow Q $ be the dual vector bundle of the associated Lie
algebroid $ A G $; then the \emph{cotangent groupoid} is a symplectic
groupoid $ T ^\ast G \rightrightarrows A ^\ast G $.  Given an element
$ \mu \in T ^\ast _g G $, the source and target are defined such that,
for all sections $ X \in \Gamma \left( A G \right) $,
\begin{equation}\label{eqn:cotangentSourceTarget}
  \bigl\langle \widetilde{ \alpha } (\mu) , X \left( \alpha (g)
    \right) \bigr\rangle  =  \bigl\langle \mu,  \overrightarrow{ X
    }(g) \bigr\rangle , \qquad
  \bigl\langle \widetilde{ \beta } (\mu) , X \left( \beta (g)
    \right) \bigr\rangle  =  \bigl\langle \mu,   \overleftarrow{ X }(g)
  \bigr\rangle .
\end{equation}
(The definition of multiplication in $ T ^\ast G $ is slightly more
intricate, and involves left and right translation by bisections;
these are discussed in \autoref{app:bisections}.)

\subsection{Discrete dynamics of Lagrangian submanifolds}
In this section, we discuss how a Lagrangian submanifold $ \Sigma
\subset \widetilde{ G } $ of a symplectic groupoid $ \widetilde{ G }
\rightrightarrows P $ gives rise to discrete dynamics.  In general,
the dynamics are only defined implicitly, as a relation, rather than
as an explicit flow map; we describe the conditions under which an
explicit dynamical flow can be given.  These dynamics can be
interpreted either as \emph{discrete Lagrangian dynamics} on $\Sigma$
or as \emph{discrete Hamiltonian dynamics} on $P$.  In the discrete
Hamiltonian case, the explicit flow is shown to be a (local) Poisson
automorphism.

\begin{definition}
  \label{def:discreteLagrangianDynamics}
  Given a symplectic groupoid $ \widetilde{ G } \rightrightarrows P $,
  let $ \Sigma \subset \widetilde{ G } $ be a Lagrangian submanifold.
  Then a sequence $ \mu _1, \ldots, \mu _n \in \widetilde{ G } $
  satisfies the \emph{discrete Lagrangian dynamics} of $\Sigma$ if $
  \mu _1 , \ldots, \mu _n \in \Sigma $ and
  \begin{equation*}
    \widetilde{ \beta } \left( \mu _k \right) = \widetilde{ \alpha }
    \left( \mu _{ k + 1 } \right) , \quad k = 1, \ldots, n - 1 ,
  \end{equation*}
  i.e., $ \mu _1, \ldots, \mu _n \in \Sigma $ forms a composable
  sequence in $ \widetilde{ G } $.
\end{definition}

The original case of discrete Lagrangian dynamics on Lie groupoids,
introduced by \citet{Weinstein1996} and elaborated further by
\citet{MaMaMa2006}, arises when $ G \rightrightarrows Q $ is a Lie
groupoid equipped with a function $ L \colon G \rightarrow \mathbb{R}
$, called the \emph{discrete Lagrangian}.  In this case, $L$ generates
a Lagrangian submanifold $ \mathrm{d} L (G) \subset T ^\ast G
\rightrightarrows A ^\ast G $ of the cotangent groupoid, and the
resulting discrete dynamics are described as follows.

\begin{theorem}
  \label{thm:unconstrained}
  Let $ G \rightrightarrows Q $ be a Lie groupoid equipped with a
  discrete Lagrangian $ L \colon G \rightarrow \mathbb{R} $. Then a
  sequence $ \mu _1 , \ldots, \mu _n \in T ^\ast G $ satisfies the
  discrete Lagrangian dynamics of $ \mathrm{d} L (G) \subset T ^\ast G
  $ if and only if
  \begin{equation*}
    \mu _k = \mathrm{d} L \left( g _k \right) \text{ for some } g _k
    \in G , \quad k = 1, \ldots, n ,
  \end{equation*}
  and the discrete Euler--Lagrange equations
  \begin{equation*}
    \overleftarrow{ X } [L] \left( g _k \right) = \overrightarrow{ X }
    [L] \left( g _{ k + 1 } \right) , \quad k = 1, \ldots, n - 1 ,
  \end{equation*}
  (cf.~\citet{MaMaMa2006}) are satisfied.
\end{theorem}

\begin{proof}
  The first condition is simply the statement that $ \mu _1 , \ldots
  ,\mu _n \in \mathrm{d} L (G) $, as required by
  \autoref{def:discreteLagrangianDynamics}. This sequence, which we
  can now express as $ \mathrm{d} L ( g _1 ) , \ldots, \mathrm{d} L (
  g _n ) $, is therefore composable in $ T ^\ast G $ if and only if
  \begin{equation*}
    ( \widetilde{ \beta } \circ \mathrm{d} L ) \left( g _k \right) =
    ( \widetilde{ \alpha } \circ \mathrm{d} L ) \left( g _{ k + 1 }
    \right) , \quad k = 1, \ldots, n - 1 .
  \end{equation*} 
  The maps $ \mathbb{F} ^- L = \widetilde{ \alpha } \circ \mathrm{d} L
  ,\ \mathbb{F} ^+ L = \widetilde{ \beta } \circ \mathrm{d} L \colon G
  \rightarrow A ^\ast G $ are called the \emph{discrete Legendre
    transforms} of $L$, and the condition above may also be written as
  $ \mathbb{F} ^+ L ( g _k ) = \mathbb{F} ^- L ( g _{ k + 1 } ) $.
  Applying the definitions of $ \widetilde{ \alpha } , \widetilde{
    \beta } $ from \eqref{eqn:cotangentSourceTarget}, this means that
  for any section $ X \in \Gamma \left( A G \right) $,
  \begin{equation*}
    \bigl\langle \mathrm{d} L \left( g _k \right) , \overleftarrow{ X } \left(
      g _k \right) \bigr\rangle = \bigl\langle \mathrm{d} L \left( g
      _{ k + 1 } \right), \overrightarrow{ X } \left( g _{ k + 1 }
    \right) \bigr\rangle , \quad k = 1, \ldots, n - 1 ,
  \end{equation*}
  i.e.,
  \begin{equation*}
    \overleftarrow{ X } [L] \left( g _k \right) = \overrightarrow{ X }
    [L] \left( g _{ k + 1 } \right) , \quad k = 1, \ldots, n - 1 ,
  \end{equation*}
  as claimed.
\end{proof}

\begin{remark}
  \citet{MaMaMa2006} showed that these discrete Euler--Lagrange
  equations are also equivalent to the sequence $ g _1 , \ldots, g _n
  \in G $ corresponding to a critical point of the action sum
  \begin{equation*}
    ( g _1 , \ldots, g _n ) \mapsto \sum _{ k = 1 } ^n L ( g _k ) ,
  \end{equation*} 
  considered over an appropriate space of admissible sequences. This
  equivalence also follows from the more general variational results
  of \autoref{sec:variational}.

  In the special case where $G$ is the pair groupoid $ Q \times Q
  \rightrightarrows Q $, this precisely recovers the discrete
  Euler--Lagrange equations,
  \begin{equation*}
    \partial _0 L \left( q _k , q _{ k + 1 } \right) + \partial _1 L
    \left( q _{ k - 1 } , q _k \right) = 0 , \qquad k = 1, \ldots, n -
    1 ,
  \end{equation*}
  as in \citet{MaWe2001}.
\end{remark}

These dynamics are implicitly defined, since they are given by a
relation on $ \widetilde{ G } $ rather than an explicitly defined map.
Restating \autoref{def:discreteLagrangianDynamics} in these terms, we
see that $ \mu _1, \ldots, \mu _n \in \widetilde{ G } $ satisfies the
discrete Lagrangian dynamics if and only if each pair of successive
elements satisfies the relation $ \left( \mu _k , \mu _{ k + 1 }
\right) \in \widetilde{ G } _2 \cap \left( \Sigma \times \Sigma
\right) $.  (In fact, such a relation may be defined by \emph{any}
subset of a groupoid, although in general, the resulting dynamics will
not preserve any notable geometric structure.)

This raises the following question: Under what conditions is this
relation, in fact, the graph of an explicit flow map $ \mu _k \mapsto
\mu _{ k + 1 } $ (at least locally), and what properties does this
flow map have?  Clearly, if the restricted source map $ \widetilde{
  \alpha } \rvert _\Sigma \colon \Sigma \rightarrow P $ is a (local)
diffeomorphism, then this flow is given by the composition $ \left(
  \widetilde{ \alpha } \rvert _\Sigma \right) ^{-1} \circ \widetilde{
  \beta } \rvert _\Sigma $.  Furthermore, if the restricted target map
$ \widetilde{ \beta } \rvert _\Sigma \colon \Sigma \rightarrow P $ is
\emph{also} a (local) diffeomorphism, then the flow is reversible, and
its (local) inverse is $ ( \widetilde{ \beta } \rvert _\Sigma ) ^{-1}
\circ \widetilde{ \alpha } \rvert _\Sigma $.  In case both the
restricted source and target maps are (local) diffeomorphisms, we say
that $\Sigma$ is a \emph{(local) Lagrangian bisection} of the
symplectic groupoid $ \widetilde{ G } $, since it is both a (local) $
\widetilde{ \alpha } $- and $ \widetilde{ \beta } $-section, as well
as a Lagrangian submanifold.  (See \autoref{app:bisections} for more
on bisections of Lie groupoids.)

In the case where $ \Sigma = \mathrm{d}L (G) $ is generated by the
discrete Lagrangian $L$, the restricted source and target maps
correspond precisely to the discrete Legendre transforms $ \mathbb{F}
^\pm L $, and the local bisection condition corresponds to regularity
of the discrete Lagrangian. (This is discussed further, and
generalized to constrained mechanics, in \autoref{sec:constraints}.)
Therefore, $ \widetilde{ \alpha } \rvert _\Sigma $ and $ \widetilde{
  \beta } \rvert _\Sigma $ can be seen as \emph{generalized discrete
  Legendre transforms}, and the question of whether they are local
diffeomorphisms generalizes the question of discrete regularity.

We now show that, in fact, if \emph{either} of the maps $ \widetilde{
  \alpha } \rvert _\Sigma $ or $ \widetilde{ \beta } \rvert _\Sigma $
is a local diffeomorphism, then both are.  That is, if the Lagrangian
submanifold $ \Sigma \subset \widetilde{ G } $ is either a local $
\widetilde{ \alpha } $- or $ \widetilde{ \beta } $-section, then it is
both, i.e., $\Sigma$ is a local Lagrangian bisection.  (Note that this
result depends crucially upon the fact that $ \widetilde{ G } $ is a
\emph{symplectic} groupoid and that $ \Sigma \subset \widetilde{ G } $
is Lagrangian; it is not true for arbitrary submanifolds of Lie
groupoids.)  We begin by proving a lemma on complementary subspaces in
a symplectic vector space; in the main theorem, this lemma is then
applied to the $ \widetilde{ \alpha } $- and $ \widetilde{ \beta }
$-vertical tangent subspaces, exploiting the fact that they are
symplectic orthogonal complements.

We note that this is a significant generalization of previous results
on the regularity of discrete Lagrangian dynamics.
\citet{Weinstein1996} first raised the question of how regularity
results for the pair groupoid $ Q \times Q $ might be generalized to
arbitrary Lie groupoids $ G \rightrightarrows Q $, and this question
was answered by \citet[Theorem 4.13]{MaMaMa2006}.  Here,
we extend this answer from the Lagrangian submanifold $ \mathrm{d} L
(G) \subset T ^\ast G $ to arbitrary Lagrangian submanifolds of
symplectic groupoids.

\begin{lemma}
  \label{lem:symplecticComplements}
  Let $ \left( Z , \omega \right) $ be a $ 2{d} $-dimensional
  symplectic vector space with subspaces $ V, W \subset Z $, and
  denote their symplectic orthogonal complements by $ V ^\omega , W
  ^\omega \subset Z $, respectively.  If $ \dim V + \dim W = 2{d} $,
  then $ \dim \left( V \cap W \right) = \dim \left( V ^\omega \cap W
    ^\omega \right) $.
\end{lemma}

\begin{proof}
  Observe that $ \left( V + W \right) ^\omega = V ^\omega \cap W
  ^\omega $.  Since the subspaces $ V + W $ and $ \left( V + W \right)
  ^\omega $ are symplectic orthogonal complements, the sum of their
  dimensions is $ 2 {d} $.  Therefore,
  \begin{align*}
    \dim \left( V ^\omega \cap W ^\omega \right) &= 2{d} - \dim \left(
      V + W \right) \\
    &= 2{d} - \dim V - \dim W + \dim \left( V \cap W \right) \\
    &= \dim \left( V \cap W \right),
  \end{align*}
  which completes the proof.
\end{proof}

\begin{theorem}
  \label{thm:lagrangianBisection}
  Let $ \widetilde{ G } \rightrightarrows P $ be a $ 2{d}$-dimensional
  symplectic groupoid, and suppose $ \Sigma \subset \widetilde{ G } $
  is a Lagrangian submanifold.  Then the restricted source map $
  \widetilde{ \alpha } \rvert _\Sigma \colon \Sigma \rightarrow P $ is
  a local diffeomorphism if and only if the restricted target map $
  \widetilde{ \beta } \rvert _\Sigma \colon \Sigma \rightarrow P $ is.
\end{theorem}

\begin{proof}
  For any $ \mu \in \Sigma $, let $ V = \ker T _\mu \widetilde{ \alpha
  } $ (i.e., the tangent space to the source fiber at $\mu$) and $ W =
  T _\mu \Sigma $.  The source and target fibers are symplectic
  orthogonal complements, so $ V ^{\widetilde{\omega}} = \ker T _\mu
  \widetilde{ \beta } $, and because the fibers have equal dimension,
  it follows that $ \dim V = \dim V ^{\widetilde{\omega}} = {d} $.
  Next, because $ \Sigma $ is a Lagrangian submanifold, this implies
  that $ W ^{\widetilde{\omega}} = W $, so $ \dim W = {d} $.
  Therefore, $ \dim V + \dim W = 2{d} $, so these subspaces of $ T
  _\mu \widetilde{ G } $ satisfy the conditions of
  \autoref{lem:symplecticComplements}.

  Finally, $ \widetilde{ \alpha } \rvert _\Sigma $ is a local
  diffeomorphism at $\mu$ if and only if the kernel
  \begin{equation*}
    \ker T _\mu \left( \widetilde{ \alpha } \rvert _\Sigma \right) =
    \ker T _\mu \widetilde{ \alpha } \cap T _\mu \Sigma = V \cap W 
  \end{equation*}
  is trivial.  Likewise, $ \widetilde{ \beta } \rvert _\Sigma $ is a
  local diffeomorphism at $\mu$ if and only if the kernel
  \begin{equation*}
    \ker T _\mu ( \widetilde{ \beta } \rvert _\Sigma )= \ker T _\mu
    \widetilde{ \beta } \cap T _\mu \Sigma = V ^{\widetilde{\omega}}
    \cap W ^{\widetilde{\omega}} 
  \end{equation*}
  is trivial.  But \autoref{lem:symplecticComplements} implies that
  these kernels have equal dimension, so one is trivial if and only if
  the other is.
\end{proof}

\begin{remark}
  Although we are only concerned with the conditions for $ \widetilde{
    \alpha } \rvert _\Sigma $ and $ \widetilde{ \beta } \rvert _\Sigma
  $ to be local diffeomorphisms, the proof of
  \autoref{thm:lagrangianBisection} also applies to any other property
  that can be described in terms of the dimension of the kernels of
  their tangent maps.  (For example: $ \widetilde{ \alpha } \rvert
  _\Sigma $ has constant rank if and only if $ \widetilde{ \beta }
  \rvert _\Sigma $ does.)
\end{remark}

Using \autoref{thm:lagrangianBisection}, we deduce that if $
\widetilde{ \alpha } \rvert _\Sigma $ is a local diffeomorphism, then
so is $ \widetilde{ \beta } \rvert _\Sigma $, and hence the discrete
Lagrangian flow map $ \left( \widetilde{ \alpha } \rvert _\Sigma
\right) ^{-1} \circ \widetilde{ \beta } \rvert _\Sigma $ is a local
automorphism on $\Sigma$.  Reversing the order of composition, it also
follows that the discrete Hamiltonian flow map $ \widetilde{ \beta }
\rvert _\Sigma \circ \left( \widetilde{ \alpha } \rvert _\Sigma
\right) ^{-1} $ is a local automorphism on $P$---and moreover, it is a
local Poisson automorphism.  To see this, consider the Poisson map $ (
\widetilde{ \alpha } , \widetilde{ \beta } ) \colon \widetilde{ G }
\rightarrow P ^{-} \times P $, $ \mu \mapsto (\widetilde{ \alpha }
(\mu) , \widetilde{ \beta } (\mu) ) $, and observe that the image of
$\Sigma$ is precisely the graph of $ \widetilde{ \beta } \rvert
_\Sigma \circ \left( \widetilde{ \alpha } \rvert _\Sigma \right) ^{-1}
$ in $ P ^{-} \times P $.  However, since $ \Sigma $ is Lagrangian,
its image under the Poisson map $ ( \widetilde{ \alpha }, \widetilde{
  \beta } ) $ is coisotropic; thus, it follows from a result of
\citet{Weinstein1988} that $ \widetilde{ \beta } \rvert _\Sigma \circ
\left( \widetilde{ \alpha } \rvert _\Sigma \right) ^{-1} $ is a
(local) Poisson automorphism on $P$, since its graph is coisotropic in
$ P ^{-} \times P $.  This argument is essentially due to
\citet{Ge1990}, who further showed that this map from (local)
Lagrangian bisections of a symplectic groupoid to (local) Poisson
automorphisms on the base is a group homomorphism.  (The group
structure that allows for composition of bisections is discussed in
\autoref{app:bisections}.)

\section{Discrete constrained Lagrangian mechanics}
\label{sec:constraints}

\subsection{Generating Lagrangian submanifolds of $T^*G$}
In this section, we will be concerned with generating a Lagrangian
submanifold $ \Sigma _L \subset T ^\ast G $ of the cotangent groupoid,
associated to a function $L$ called the \emph{discrete Lagrangian}.
This is a particular example of the formalism introduced in the
previous section.  As shown in \autoref{thm:unconstrained}, one way to
do this is by defining a discrete Lagrangian $ L \colon G \rightarrow
\mathbb{R} $ and taking $ \Sigma _L = \mathrm{d} L (G) $; this is the
case considered in previous work on discrete Lagrangian mechanics,
including \citet{Weinstein1996} and \citet{MaMaMa2006}.

Setting aside the groupoid structure for the moment, this approach
exploits the fact that for any manifold $M$ and function $ L \colon M
\rightarrow \mathbb{R} $, the submanifold $ \mathrm{d} L (M) \subset T
^\ast M $ is Lagrangian.  However, there is a more general
construction due to \citet{SnTu1972} (see also \citet{Tulczyjew1977}),
which we will use to generalize the earlier approach to discrete
mechanics.

\begin{theorem}[\citet{SnTu1972}]
  \label{thm:tulczyjew}
  Let $M$ be a smooth manifold, $ N \subset M $ a submanifold, and $ L
  \colon N \rightarrow \mathbb{R} $.  Then
  \begin{multline*}
    \Sigma _L = \{ p \in T ^\ast M \;\vert\; \pi _M (p)
      \in N \text{ and } \left\langle p, v \right\rangle =
      \left\langle \mathrm{d} L , v \right\rangle \\
      \text{for all } v \in T N \subset T M \text{ such that } \tau _M
      (v) = \pi _M (p) \} 
  \end{multline*}
  is a Lagrangian submanifold of $ T ^\ast M $.
\end{theorem}

(Here, $ \pi _M \colon T ^\ast M \rightarrow M $ and $ \tau _M \colon
T M \rightarrow M $ denote the cotangent and tangent bundle
projections, respectively.)  In the special case $ N = M $, this gives
the familiar Lagrangian submanifold $ \Sigma _L = \mathrm{d} L (M)
\subset T ^\ast M $.

Turning back to the groupoid formulation, this allows one to define a
discrete Lagrangian on some constraint submanifold $ N \subset G $,
rather than necessarily on all of $G$.  This realization motivates the
following definition.

\begin{definition}
  A \emph{discrete constrained Lagrangian system} consists of a triple
  $ \left( G, N , L \right) $, where $ G \rightrightarrows Q $ is a
  Lie groupoid, $ N \subset G $ is a submanifold, and $ L \colon N
  \rightarrow \mathbb{R} $ is a function called the discrete
  Lagrangian.
\end{definition}

It follows immediately from \autoref{thm:tulczyjew} that a discrete
constrained Lagrangian system generates a Lagrangian submanifold $
\Sigma _L \subset T ^\ast G $ of the cotangent groupoid $ T ^\ast G
\rightrightarrows A ^\ast G $.  The relationship among these spaces is
shown in the following diagram:
\begin{equation*}
  \xymatrix{ & \Sigma _L \mathrel{\ar@{^(->}[r]} \ar[d]& T ^\ast G
    \ar@<.5ex>[r] ^{ \widetilde{ \alpha } }  \ar@<-.5ex>[r]_{
      \widetilde{ \beta } }  \ar[d]& A ^\ast G \ar[d]\\ \mathbb{R}  &
    \ar[l]_L    N  \mathrel{\ar@{^(->}[r]}& G \ar@<.5ex>[r] ^\alpha
    \ar@<-.5ex>[r] _\beta  & Q  } 
\end{equation*}

\begin{definition}
  Let $ \left( G, N, L \right) $ be a discrete constrained Lagrangian
  system.  Then define the \emph{discrete Legendre transformations} $
  \mathbb{F} ^\pm L \colon \Sigma _L \rightarrow A ^\ast G $ to be the
  restricted source and target maps,
  \begin{equation*}
    \mathbb{F}  ^- L = \widetilde{ \alpha } \rvert _{ \Sigma _L },
    \qquad \mathbb{F}  ^+ L = \widetilde{ \beta } \rvert _{ \Sigma _L
    } .
  \end{equation*}
\end{definition}

\begin{corollary}
  The discrete Legendre transformation $ \mathbb{F} ^- L $ is a local
  diffeomorphism if and only if $ \mathbb{F} ^+ L $ is.
\end{corollary}

\begin{proof}
  Direct application of \autoref{thm:lagrangianBisection}.
\end{proof}

\begin{definition}
  A discrete constrained Lagrangian system $ \left( G, N, L \right) $
  is said to be \emph{regular} if $ \mathbb{F} ^\pm L $ are local
  diffeomorphisms, and \emph{hyperregular} if they are global
  diffeomorphisms.
\end{definition}

Now, given a trajectory $ \mu _1, \ldots, \mu _n \in T ^\ast G $, it
follows from \autoref{def:discreteLagrangianDynamics} that this
satisfies the discrete Lagrangian dynamics of $ \Sigma _L $ when $ \mu
_1 , \ldots, \mu _n \in \Sigma _L $ and
\begin{equation*}
  \mathbb{F}  ^+ L \left( \mu _k \right) = \mathbb{F}  ^- L \left( \mu
    _{ k + 1 } \right) , \quad k = 1, \ldots, n - 1 .
\end{equation*}
Hence, if $ \left( G, N , L \right) $ is regular, then the discrete
Lagrangian flow map is given by the local automorphism $ \left(
  \mathbb{F} ^- L \right) ^{-1} \circ \mathbb{F} ^+ L $ on $ \Sigma _L
$, while the discrete Hamiltonian flow map is given by the local
Poisson automorphism $ \mathbb{F} ^+ L \circ \left( \mathbb{F} ^- L
\right) ^{-1} $ on $ A ^\ast G $.  In the special case $ N = G $, $
\Sigma _L = \mathrm{d} L (G) $, this is in agreement with the
formulation of \citet{Weinstein1996} and \citet{MaMaMa2006}, as
discussed in \autoref{thm:unconstrained}.

\begin{remark}
  \label{rmk:subgroupoid}
  It should be emphasized that this formalism does not require $N$ to
  be a \emph{subgroupoid} of $G$, but only a \emph{submanifold}.  If $
  N \subset G $ is indeed a subgroupoid, then one can simply reduce to
  the unconstrained dynamics on $ N \rightrightarrows Q $, which are
  given by the Lagrangian submanifold $ \mathrm{d} L (N) \subset T
  ^\ast N \rightrightarrows A ^\ast N $.  Therefore, in the
  subgroupoid case, $N$ can be thought of as specifying \emph{discrete
    integrable constraints} on $G$.  By contrast, when $N$ is not a
  subgroupoid, it can be thought of as specifying \emph{discrete
    non-integrable constraints} on $G$.

  This is consistent with the definition of integrable and
  non-integrable constraints for continuous Lagrangian systems.  For
  example, a constraint distribution $ \Delta \subset T Q $ is
  integrable precisely when $\Delta$ is closed under the Jacobi--Lie
  bracket on $ T Q $, i.e., when $ \Delta $ is a \emph{subalgebroid}
  of the tangent Lie algebroid $ T Q $.  Just as continuous integrable
  constraints correspond to Lie subalgebroids, it is natural to think
  of discrete integrable constraints as corresponding to Lie
  subgroupoids.

  Finally, we note that this approach to \emph{non-integrable}
  constrained mechanics should not be confused with
  \emph{nonholonomic} mechanics. Indeed, in nonholonomic mechanics,
  one would generally not have a symplectic or Poisson flow, and thus
  the associated graph would not be a Lagrangian submanifold. The
  approach presented here is more closely related to the so-called
  \emph{vakonomic} approach to non-integrable constraints, which is
  commonly applied to problems in optimal control, for example.
\end{remark}

\subsection{Affine bundle structure of \texorpdfstring{$ \Sigma _L
    $}{\textbackslash Sigma\_L}} For any discrete constrained
Lagrangian system $ \left( G, N , L \right) $, the Lagrangian
submanifold $ \Sigma _L \subset T ^\ast G $ is also a bundle over $N$.
More precisely, taking the projection to be the restriction of the
cotangent bundle projection $ \pi _G \rvert _{ \Sigma _L } \colon
\Sigma _L \rightarrow N $, we obtain an affine bundle whose associated
vector bundle is $ \nu ^\ast N $, the conormal bundle of $N$ in $G$.
To see this, note that for any $ \mu \in \Sigma _L $ and $ \Lambda \in
\nu ^\ast N $ at the same basepoint $ \pi _G (\mu) = \pi _G (\Lambda)
= g $, we have
\begin{equation*}
  \left\langle \mu + \Lambda , v \right\rangle = \left\langle \mu , v
  \right\rangle = \left\langle \mathrm{d} L (g), v \right\rangle 
\end{equation*}
for all $ v \in T _g N $, and thus $ \mu + \Lambda \in \Sigma _L $.
(For more details, see \citet{LiMa1987} and references therein.)

Therefore, $ \Sigma _L $ is isomorphic to $ \nu ^\ast N $, but
generally not in any canonical way.  To choose a particular
isomorphism, if one so desires, it suffices to specify a distinguished
section $ \sigma \colon N \rightarrow \Sigma _L $, and then
\begin{equation*}
  \Sigma _L = \left\{ \sigma (g) + \Lambda \;\middle\vert\; \Lambda
    \in \nu ^\ast N,\ g = \pi _G (\Lambda) \right\} \cong \nu ^\ast N .
\end{equation*}
In particular, suppose that $ \widehat{L} $ is an extension of $L$ to
a neighborhood of $N$ in $G$, so that $ L = \widehat{L} \rvert _N $.
Then this defines the distinguished section $ \sigma = \mathrm{d}
\widehat{L} \rvert _N $, and hence
\begin{equation*}
  \Sigma _L = \bigl\{ \mathrm{d} \widehat{L} (g) + \Lambda
    \;\bigm\vert\; \Lambda \in \nu ^\ast N,\ g = \pi _G (\Lambda)
  \bigr\} .
\end{equation*}
Applying the definitions of $ \widetilde{ \alpha } $ and $ \widetilde{
  \beta } $, it follows that the discrete Legendre transformations can
be written as
\begin{equation*}
  \bigl\langle \mathbb{F}  ^- L \bigl( \mathrm{d} \widehat{L} (g) +
  \Lambda \bigr) , X \left( \alpha (g) \right) \bigr\rangle =
  \bigl\langle  \mathrm{d} \widehat{L} (g) + \Lambda ,
  \overrightarrow{ X } (g) \bigr\rangle = \overrightarrow{ X }
  [\widehat{L}] (g) + \bigl\langle \Lambda , \overrightarrow{ X } (g)
  \bigr\rangle ,
\end{equation*}
and similarly,
\begin{equation*}
  \bigl\langle \mathbb{F}  ^+ L \bigl( \mathrm{d} \widehat{L} (g) +
  \Lambda \bigr) , X \left( \beta (g) \right) \bigr\rangle =
  \overleftarrow{ X } [\widehat{L}] (g) + \bigl\langle \Lambda,
  \overleftarrow{ X } (g) \bigr\rangle ,  
\end{equation*}
for all sections $ X \in \Gamma \left( A G \right) $.

Now, given a sequence $ \mu _1, \ldots, \mu _n \in T ^\ast G $, let $
g _k = \pi _G \left( \mu _k \right) $ and take $ \Lambda _k = \mu _k -
\mathrm{d} \widehat{L} \left( g _k \right) $ for $ k = 1, \ldots, n $.
Therefore, this is a solution of the discrete Lagrangian dynamics when
$ \Lambda _1, \ldots, \Lambda _n \in \nu ^\ast N $ and
\begin{equation}
\label{eqn:delLambda}
  \overleftarrow{ X } [\widehat{L}] \left( g _k \right)  +
  \bigl\langle \Lambda _k , \overleftarrow{ X } \left( g _k \right)
  \bigr\rangle  =   \overrightarrow{ X } [\widehat{L}] \left( g _{ k +
      1 } \right)  + \bigl\langle \Lambda _{ k + 1 } ,
  \overrightarrow{  X } \left( g _{ k + 1 } \right) \bigr\rangle ,
\end{equation}
for all sections $ X \in \Gamma \left( A G \right) $ and $ k = 1,
\ldots, n - 1 $.  In the special case $ N = G $, observe that $ \nu
^\ast N $ is simply the zero section of $ T ^\ast G $, which is
isomorphic to $G$ itself; hence, the dynamics reduce to the
unconstrained discrete Euler--Lagrange equations of
\autoref{thm:unconstrained}.

\subsection{Lagrange multipliers}
\label{sec:lagrangeMultipliers}
Suppose that the constraint submanifold $ N \subset G $ is defined by
\begin{equation*}
  N = \left\{ g \in G \;\middle\vert\; \phi ^a (g) = 0 ,\ a \in A
  \right\} ,
\end{equation*}
where $ \left\{ \phi ^a \right\} _{ a \in A } $ is a family of real
functions defined in a neighborhood of $N$ and $A$ is an index set.
It follows that $ \left\{ \mathrm{d} \phi ^a \rvert _N \right\} _{ a
  \in A } $ is a basis of sections of the conormal bundle $ \nu ^\ast
N $.  Hence, a section $ \Lambda $ of the conormal bundle can be
written
\begin{equation*}
  \Lambda = \lambda _a \,\mathrm{d} \phi ^a \rvert _N ,
\end{equation*}
where the real functions $ \lambda _a $ on $N$ are called
\emph{Lagrange multipliers}.  (Here, we have used the Einstein
summation notation to indicate that we are summing over $ a \in A $.)
In fact, since $ \phi ^a \rvert _N = 0 $, we can deduce that
\begin{equation*}
  \Lambda = \mathrm{d} \left( \lambda _a \phi ^a \rvert _N \right) .
\end{equation*}
Now, as before, suppose that the discrete Lagrangian $L \colon N
\rightarrow \mathbb{R} $ is the restriction to $N$ of a real function
$ \widehat{L} $ on $G$.  Therefore, an element $ \mu \in \Sigma _L $,
with $ g = \pi _G (\mu) $, can be written as
\begin{equation*}
  \mu = \mathrm{d} \widehat{L} (g) + \mathrm{d} \left( \lambda _a \phi
    ^a \right) (g) = \mathrm{d} \bigl( \widehat{L} + \lambda _a \phi
    ^a \bigr) (g) \in \Sigma _L .
\end{equation*}
In this sense, $ \Sigma _L $ can be seen as the space consisting of
elements $ g \in N $, together with the Lagrange multipliers $\lambda$
constraining $g$ to $N$.

Thus, if $ \left( g, \lambda \right) \in \Sigma _L $, we have
\begin{equation*} 
  \bigl\langle \mathbb{F} ^- L \left( g, \lambda \right) , X
  \left( \alpha (g) \right) \bigr\rangle = \bigl\langle \mathrm{d}
  \bigl( \widehat{L} + \lambda _a \phi ^a \bigr) (g) ,
  \overrightarrow{ X } (g) \bigr\rangle  = \overrightarrow{ X } \bigl[
  \widehat{L} + \lambda _a \phi ^a \bigr] (g),
\end{equation*} 
and likewise
\begin{equation*}
  \bigl\langle \mathbb{F}  ^+ L \left( g, \lambda \right) , X \left( \beta (g)
  \right) \bigr\rangle = \overleftarrow{ X } \bigl[ \widehat{L} +
  \lambda _a \phi ^a \bigr] (g) ,
\end{equation*}
for all sections $ X \in \Gamma \left( A G \right) $.  Consequently,
let $ \left( g _1, \lambda _1 \right) , \ldots , \left( g _n , \lambda
  _n \right) $ be a sequence of groupoid elements and Lagrange
multipliers.  Then this is a solution of the discrete Lagrangian
dynamics when $ g _1, \ldots, g _n \in N $, i.e.,
\begin{equation*}
  \phi ^a \left( g _k \right) = 0 \text{ for all } a \in A, \quad k =
  1, \ldots, n ,
\end{equation*}
and when
\begin{equation*}
  \overleftarrow{ X } \bigl[ \widehat{L} + \left( \lambda _k \right)
  _a \phi ^a \bigr]  \left( g _k \right)  = \overrightarrow{ X }
  \bigl[ \widehat{L} + \left( \lambda _{ k + 1 } \right) _a \phi ^a
  \bigr] \left( g _{ k + 1 } \right) , \quad k = 1, \ldots, n - 1 ,
\end{equation*}
for all sections $ X \in \Gamma \left( A G \right) $.

\section{Morphisms, reduction, and Noether symmetries}
\label{sec:morphisms}

\subsection{Groupoid morphisms and reduced dynamics}
In order to study reduction of discrete constrained Lagrangian
systems, we first recall the definition of a morphism of Lie
groupoids.  After this, we introduce the slightly more specialized
notion of a morphism for discrete constrained Lagrangian systems,
which preserves not only the groupoid structure, but also that of the
constraint submanifolds and the Lagrangian functions.

\begin{definition}
  Given two Lie groupoids, $ G \rightrightarrows Q $ and $ G ^\prime
  \rightrightarrows Q ^\prime $, a smooth map $ \Phi \colon G
  \rightarrow G ^\prime $ is a \emph{morphism of Lie groupoids} if,
  for every composable pair $ \left( g, h \right) \in G _2 $, it
  satisfies $ \left( \Phi (g), \Phi (h) \right) \in G ^\prime _2 $ and
  $ \Phi \left( gh \right) = \Phi (g) \Phi (h) $.
\end{definition}

\begin{definition}
  Given two discrete constrained Lagrangian systems, denoted $ \left(
    G, N, L \right) $ and $ \left( G ^\prime, N ^\prime , L ^\prime
  \right) $, a smooth map $ \Phi \colon G \rightarrow G ^\prime $ is a
  \emph{morphism of discrete constrained Lagrangian systems} if it is
  a morphism of Lie groupoids, and additionally, it satisfies $ N =
  \Phi ^{-1} \left( N ^\prime \right) $ (i.e., $ g \in N
  \Leftrightarrow \Phi (g) \in N ^\prime $) and $ L = L ^\prime \circ
  \Phi $.
\end{definition}

\begin{example}
  Suppose that $ \left( G, N , L \right) $ and $ \left( G ^\prime , N
    ^\prime , L ^\prime \right) $ are discrete constrained Lagrangian
  systems, with constraint submanifolds  defined by
  \begin{equation*}
    N = \left\{ g \in G \;\middle\vert\; \phi ^a (g) = 0 ,\ a \in A
    \right\} , \qquad N ^\prime = \left\{ g ^\prime \in G ^\prime
      \;\middle\vert\; \phi ^{ \prime a } \left( g ^\prime \right)  =
      0 , \ a  \in A \right\} .
  \end{equation*}
  If $ \Phi \colon G \rightarrow G ^\prime $ is a morphism of Lie
  groupoids satisfying $ \phi ^a = \phi ^{ \prime a } \circ \Phi $,
  then clearly $ \phi ^a (g) = 0 \Leftrightarrow \phi ^{ \prime a }
  \left( \Phi (g) \right) = 0 $, so $ g \in N \Leftrightarrow \Phi (g)
  \in N ^\prime $.  If, furthermore, we have $ L = L ^\prime \circ
  \Phi $, then this implies that $ \Phi $ is a morphism of the
  discrete constrained Lagrangian systems.
\end{example}

A morphism $ \Phi \colon G \rightarrow G ^\prime $ of Lie groupoids
induces a smooth map $ \Phi _0 \colon Q \rightarrow Q ^\prime $ on the
base, which satisfies
\begin{equation*}
  \alpha ^\prime \circ \Phi = \Phi _0 \circ \alpha , \quad \beta
  ^\prime \circ \Phi = \Phi _0 \circ \beta , \quad \Phi \circ \epsilon
  = \epsilon ^\prime \circ \Phi _0 , \quad \Phi \circ i = i ^\prime
  \circ \Phi _0 .
\end{equation*}
Moreover, $\Phi$ induces a morphism $ A \Phi \colon A G \rightarrow A
G ^\prime $ of the corresponding Lie algebroids, and
\begin{gather*}
  \overleftarrow{ A \Phi (v) } (g) = T \ell _{ \Phi (g) } \left( A
    \Phi (v) \right) = T \Phi \left( T \ell _g (v) \right) = T \Phi
  \left( \overleftarrow{ v } (g) \right) ,\\
  \overrightarrow{ A \Phi (w) } (g) = - T \left( r _{ \Phi (g) } \circ
    i ^\prime \right) \left( A \Phi (w) \right) = T \Phi \left( - T
    \left( r _g \circ i \right) (w) \right) = T \Phi \left(
    \overrightarrow{ w } (g) \right) ,
\end{gather*}
for any $ g \in G $ and any $ v \in A _{ \beta (g) } G $, $ w \in A _{
  \alpha (g) } G $.  Consequently, two sections $ X \in \Gamma \left(
  A G \right) $ and $ X ^\prime \in \Gamma \left( A G ^\prime \right)
$ satisfy $ A \Phi \circ X = X ^\prime \circ \Phi _0 $ if and only if
$ T \Phi \circ \overleftarrow{ X } = \overleftarrow{ X ^\prime } \circ
\Phi $, or equivalently, $ T \Phi \circ \overrightarrow{ X } =
\overrightarrow{ X ^\prime } \circ \Phi $.  (See, e.g.,
\citet[p.~125]{Mackenzie2005}.)  That is, two sections are ``$ A \Phi
$-related'' if and only if their corresponding left-invariant (and
right-invariant) vector fields are $ \Phi $-related.  It is now
possible to define the dual notions of these relations on the
cotangent groupoid and its base.

\begin{definition}
  Two covectors $ \mu \in T _g ^\ast G $ and $ \mu ^\prime \in T ^\ast
  _{ \Phi (g) } G ^\prime $ are \emph{$ \mathit{ \Phi ^\ast }
    $-related} if $ \left\langle \mu , v \right\rangle = \left\langle
    \mu ^\prime , T \Phi (v) \right\rangle $ for all $ v \in T _g G $.
  Similarly, $ p \in A ^\ast _q G $ and $ p ^\prime \in A ^\ast _{
    \Phi _0 (q) } G ^\prime $ are \emph{$ \mathit{A ^\ast \Phi}
    $-related} if $ \left\langle p, v \right\rangle = \left\langle p
    ^\prime , A \Phi (v) \right\rangle $ for all $ v \in A _q G $.
\end{definition}

We now have the necessary equipment to state and prove the main
theorem on the reduction of discrete constrained Lagrangian systems,
which states how the structure of the cotangent groupoid and
Lagrangian submanifold transform under morphisms.  As an immediate
corollary, we see that the Lagrangian dynamics of $ \left( G, N , L
\right) $ reduce to those of $ \left( G ^\prime, N ^\prime , L ^\prime
\right) $, with respect to $ \Phi ^\ast $-relatedness.  (The content
of this theorem is shown diagrammatically in \autoref{fig:reduction}.)

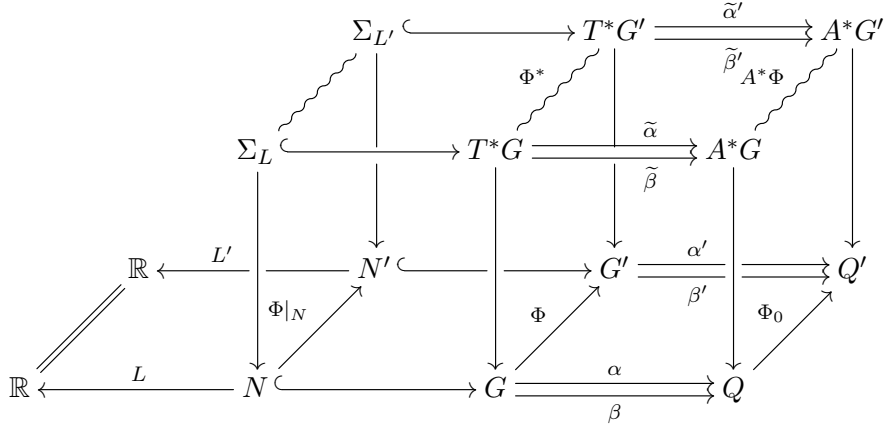
\begin{figure}
  \begin{equation*}
    \xymatrix @!=1.75pc{
      &&& \Sigma _{L ^\prime} \,\ar@{^(->}[rr] \ar'[d][dd] \ar@{~}[ld]
      &&
      T ^\ast G ^\prime \ar@<.5ex>[rr]^{ \widetilde{ \alpha } ^\prime }
      \ar@<-.5ex>[rr] _{ \widetilde{ \beta } ^\prime } \ar'[d][dd]
      \ar@{~}[ld] _{ \Phi ^\ast  }
      && A ^\ast G  ^\prime \ar[dd]
      \ar@{~}[ld] _{A ^\ast \Phi } \\
      &&\Sigma _L  \,\ar@{^(->}[rr] \ar[dd] &&
      T ^\ast G \ar@<.5ex>[rr]^(.65){\widetilde{\alpha}}
      \ar@<-.5ex>[rr]_(.65){\widetilde{\beta}}
      \ar[dd]&&
      A ^\ast G \ar[dd] \\
      &\mathbb{R} && N ^\prime \,\ar@{^(->}'[r][rr]
      \ar'[l][ll]_(.3){L ^\prime}
      && G ^\prime \ar@{-}@<.5ex>[r] ^(.7){\alpha  ^\prime }
      \ar@{-}@<-.5ex>[r]_(.7){ \beta ^\prime }  &
      \ar@<.5ex>[r] \ar@<-.5ex>[r] & Q ^\prime \\
      \mathbb{R} \ar@{=}[ur]&&N \,\ar@{^(->}[rr] \ar[ur]^{\Phi \rvert
        _N } \ar[ll]_L &&
      G \ar@<.5ex>[rr]^\alpha \ar@<-.5ex>[rr]_\beta  \ar[ur] ^\Phi
      &&  Q \ar[ur] ^{ \Phi _0 } \\
    }
  \end{equation*}
  \caption{A diagrammatic depiction of \autoref{thm:dclsMorphism},
    showing the maps and relations involved in the reduction of
    discrete constrained Lagrangian systems.\label{fig:reduction}}
\end{figure}

\begin{theorem}
  \label{thm:dclsMorphism}
  Let $ \Phi \colon G \rightarrow G ^\prime $ be a morphism of the
  discrete constrained Lagrangian systems $ \left( G, N , L \right) $
  and $ \left( G ^\prime, N ^\prime , L ^\prime \right) $, and suppose
  $ \mu \in T ^\ast G $ and $ \mu ^\prime \in T ^\ast G ^\prime $ are
  $ \Phi ^\ast $-related.  Then, the following are true:
  \begin{enumerate}
  \item if $ \mu ^\prime \in \Sigma _{ L ^\prime } $, then $ \mu \in
    \Sigma _L $;
  \item the sources $ \widetilde{ \alpha } (\mu) \in A ^\ast G $ and $
    \widetilde{ \alpha } ^\prime \left( \mu ^\prime \right) \in A
    ^\ast G ^\prime $ are $ A ^\ast \Phi $-related;
  \item the targets $ \widetilde{ \beta } (\mu) \in A ^\ast G $ and $
    \widetilde{ \beta } ^\prime \left( \mu ^\prime \right) \in A ^\ast
    G ^\prime $ are $ A ^\ast \Phi $-related.
  \end{enumerate}
\end{theorem}

\begin{proof}
  To prove (1), we begin by denoting $ g = \pi _G (\mu) $ and $ g
  ^\prime = \pi _{ G ^\prime } \left( \mu ^\prime \right) $; since $
  \mu $ and $ \mu ^\prime $ are $ \Phi ^\ast $-related, we have $ \Phi
  (g) = g ^\prime $.  If $ \mu ^\prime \in \Sigma _{L ^\prime } $,
  then $ \Phi (g) = g ^\prime \in N ^\prime $, so $ g \in \Phi ^{-1}
  \left( N ^\prime \right) = N $.  Furthermore, for all $ v \in T _g N
  $,
  \begin{multline*}
    \left\langle \mu , v \right\rangle = \left\langle \mu ^\prime , T
      \Phi (v) \right\rangle =
    \left\langle \mathrm{d} L ^\prime , T \Phi (v) \right\rangle \\
    = \left\langle \Phi ^\ast \left( \mathrm{d} L ^\prime \right) , v
    \right\rangle = \left\langle \mathrm{d} \left( L ^\prime \circ
        \Phi \right) , v \right\rangle = \left\langle \mathrm{d} L , v
    \right\rangle ,
  \end{multline*}
  and thus $ \mu \in \Sigma _L $.

  To prove (2), suppose that $ v \in A _{ \alpha (g) } G $.  Then
  \begin{equation*}
    \bigl\langle \widetilde{ \alpha } (\mu) , v \bigr\rangle =
    \bigl\langle \mu , \overrightarrow{ v } \bigr\rangle =
    \bigl\langle \mu ^\prime, T \Phi ( \overrightarrow{ v }
    ) \bigr\rangle = \bigl\langle \mu ^\prime , \overrightarrow{ A
      \Phi (v) } \bigr\rangle = \bigl\langle \widetilde{ \alpha }
    ^\prime \left( \mu ^\prime \right) , A \Phi (v) \bigr\rangle ,
  \end{equation*}
  so $ \widetilde{ \alpha } (\mu) $ and $ \widetilde{ \alpha }
  ^\prime \left( \mu ^\prime \right) $ are $ A ^\ast \Phi $-related.
  Likewise, to prove (3), let $ v \in A _{ \beta (g) } G $.  Then
  \begin{equation*}
    \bigl\langle \widetilde{ \beta } (\mu) , v \bigr\rangle =
    \bigl\langle \mu , \overleftarrow{ v } \bigr\rangle =
    \bigl\langle \mu ^\prime, T \Phi ( \overleftarrow{ v }
    ) \bigr\rangle = \bigl\langle \mu ^\prime , \overleftarrow{ A
      \Phi (v) } \bigr\rangle = \bigl\langle \widetilde{ \beta }
    ^\prime \left( \mu ^\prime \right) , A \Phi (v) \bigr\rangle ,
  \end{equation*}
  and thus $ \widetilde{ \beta } (\mu) $ and $ \widetilde{ \beta }
  ^\prime \left( \mu ^\prime \right) $ are $ A ^\ast \Phi $-related.
\end{proof}

\begin{corollary}
  \label{cor:dclsMorphism}
  Let $ \Phi \colon G \rightarrow G ^\prime $ be a morphism of the
  discrete constrained Lagrangian systems $ \left( G, N , L \right) $
  and $ \left( G ^\prime, N ^\prime , L ^\prime \right) $.  If $ \mu
  _1 ^\prime , \ldots, \mu _n ^\prime \in T ^\ast G ^\prime $
  satisfies the discrete Lagrangian dynamics of $ \left( G ^\prime, N
    ^\prime , L ^\prime \right) $, then any $ \Phi ^\ast $-related
  sequence $ \mu _1, \ldots, \mu _n \in T ^\ast G $ satisfies the
  discrete Lagrangian dynamics of $ \left( G, N , L \right) $.
\end{corollary}

\begin{proof}
  By \autoref{thm:dclsMorphism}, $ \mu _k ^\prime \in \Sigma _{ L
    ^\prime } $ implies $ \mu _k \in \Sigma _L $ for $ k = 1, \ldots,
  n $.  Furthermore, for any $ v \in A _{ \beta \left( g _k \right) }
  G = A _{ \alpha \left( g _{ k + 1 } \right) } G $, we have
  \begin{equation*}
    \bigl\langle \widetilde{ \beta } \left( \mu _k \right) , v
    \bigr\rangle = \bigl\langle \widetilde{ \beta } ^\prime \left( \mu
      _k ^\prime \right) , A \Phi (v) \bigr\rangle = \bigl\langle
    \widetilde{ \alpha } ^\prime \left( \mu _{ k + 1 } ^\prime \right) , A
    \Phi (v) \bigr\rangle = \bigl\langle \widetilde{ \alpha } \left(
      \mu _{ k + 1 } \right) , v \bigr\rangle ,
  \end{equation*}
  so $ \widetilde{ \beta } \left( \mu _k \right) = \widetilde{ \alpha
  } \left( \mu _{ k + 1 } \right) $ for $ k = 1, \ldots, n - 1
  $. Thus, $ \mu _1, \ldots, \mu _n $ satisfies the discrete
  Lagrangian dynamics of $ \left( G, N , L \right) $.
\end{proof}

\begin{remark}
  If $ \Phi \colon G \rightarrow G ^\prime $ is a submersion, then the
  converse is true as well.  Since every $ v ^\prime \in T _{ \Phi (g)
  } N ^\prime $ can be written as $ v ^\prime = T \Phi (v) $ for $ v
  \in T _g N $, part (1) of \autoref{thm:dclsMorphism} can be
  strengthened to $ \mu \in \Sigma _L \Leftrightarrow \mu ^\prime \in
  \Sigma _{ L ^\prime } $.  Furthermore, $ A\Phi $ is a fiberwise
  surjection if (and only if) $ \Phi $ is a submersion
  (\citet[Proposition 3.5.15]{Mackenzie2005}).  Therefore, with this
  additional assumption, \autoref{cor:dclsMorphism} can be
  strengthened to say that the sequence $ \mu _1 , \ldots, \mu _n \in
  T ^\ast G $ satisfies the discrete Lagrangian dynamics if and only
  if the $ \Phi ^\ast $-related sequence $ \mu _1 ^\prime, \ldots, \mu
  _n ^\prime \in T ^\ast G ^\prime $ does.
\end{remark}

\subsection{Noether symmetries and constants of the motion} Next, we
extend the notion of discrete Noether symmetries and constants of the
motion, as defined for unconstrained systems on Lie groupoids in
\citet{MaMaMa2006}, to apply to discrete constrained Lagrangian
systems.

\begin{definition}
  A section $ X \in \Gamma \left( A G \right) $ is said to be a
  \emph{Noether symmetry} of the discrete constrained Lagrangian
  system $ \left( G, N , L \right) $ if there exists a function $ f
  \in C ^\infty (Q) $ such that
  \begin{equation*}
    \bigl\langle \widetilde{ \alpha } (\mu)
    , X \left( \alpha (g) \right) \bigr\rangle + f \left( \alpha (g)
    \right)  =     \bigl\langle \widetilde{ \beta } (\mu) , X \left( \beta (g)
    \right) \bigr\rangle + f \left( \beta (g)
    \right) 
  \end{equation*}
  for all $ \mu \in \Sigma _L $, where we denote $ g = \pi _G (\mu) $.
\end{definition}

For each Noether symmetry of a discrete constrained Lagrangian system,
there is a corresponding constant of the motion, which is preserved by
the Lagrangian dynamics; this is the discrete version of Noether's
theorem.

\begin{theorem}
  If $X \in \Gamma \left( A G \right) $ is a Noether symmetry of the
  discrete constrained Lagrangian system $ \left( G, N , L \right) $,
  then the function $ F _X \colon \Sigma _L \rightarrow \mathbb{R} $
  defined by
  \begin{equation*}
    F _X (\mu) = \left\langle \widetilde{ \alpha } (\mu) , X \left(
        \alpha (g) \right) \right\rangle + f \left( \alpha (g) \right)
    = \bigl\langle \widetilde{ \beta } (\mu) , X \left( \beta (g)
    \right) \bigr\rangle + f \left( \beta (g)
    \right) ,
  \end{equation*}
  where $ g = \pi _G (\mu) $, is a constant of the motion. That is, if
  $ \mu _1, \ldots, \mu _n \in T ^\ast G $ satisfies the discrete
  Lagrangian dynamics, then $ F _X \left( \mu _k \right) = F _X \left(
    \mu _{ k + 1 } \right) $ for $ k = 1, \ldots, n - 1 $.
\end{theorem}

\begin{proof}
  If $ \mu _1, \ldots, \mu _n \in T ^\ast G $ satisfies the discrete
  Lagrangian dynamics, then by definition, $ \mu _k \in \Sigma _L $
  for $ k = 1, \ldots, n $ and $ \widetilde{ \beta } \left( \mu _k
  \right) = \widetilde{ \alpha } \left( \mu _{ k + 1 } \right) $ for $
  k = 1, \ldots, n - 1 $.  Therefore, 
  \begin{align*}
    F _X \left( \mu _k \right) &= \bigl\langle \widetilde{ \beta }
    \left( \mu _k \right) , X \left( \beta \left( g _k \right) \right)
    \bigr\rangle + f \left( \beta \left( g _k \right) \right) \\
    &= \bigl\langle \widetilde{ \alpha } \left( \mu _{k+1} \right) , X
    \left( \alpha \left( g _{k+1} \right) \right) \bigr\rangle + f
    \left( \alpha \left( g _{k+1} \right) \right) \\
    &= F _X \left( \mu _{ k + 1 } \right) ,
  \end{align*}
  for $ k = 1 , \ldots, n - 1 $, which completes the proof.
\end{proof}

When $ f = 0 $, we say $L$ is \emph{invariant} with respect to $X$,
and the conserved quantity is simply 
\begin{equation*}
  F _X (\mu) = \left\langle \widetilde{ \alpha } (\mu) , X \left(
      \alpha (g) \right) \right\rangle = \bigl\langle \widetilde{
    \beta } (\mu) , X \left( \beta (g) \right) \bigr\rangle =
  \bigl\langle \mathbb{F}  ^\pm L (\mu) , X \bigr\rangle .
\end{equation*}
More generally, when $ f \neq 0 $, we say $L$ is
\emph{quasi-invariant} with respect to $X$. In continuous Lagrangian
mechanics, this is analogous to the fact that Noether's theorem holds
not only for symmetries, which leave the action invariant, but also
for quasi-symmetries, which change the action by a boundary term; this
extra boundary term corresponds, here, to the function $f$.

As an illustration, consider the unconstrained case $ N = G $. From
the definition above, $ X \in \Gamma ( A G ) $ is a Noether symmetry
when
\begin{equation*}
  \bigl\langle \widetilde{ \alpha } \circ \mathrm{d} L (g) 
  , X \left( \alpha (g) \right) \bigr\rangle + f \left( \alpha (g)
  \right)  = \bigl\langle \widetilde{ \beta } \circ \mathrm{d} L (g) ,
  X \left( \beta (g) \right) \bigr\rangle + f \left( \beta (g) \right) ,
\end{equation*} 
for all $ g \in G $.  Applying the definitions of $ \widetilde{ \alpha
} $ and $ \widetilde{ \beta } $, this is equivalent to
\begin{equation*}
  \overrightarrow{ X } [L] (g) + f \bigl( \alpha (g) \bigr) =
  \overleftarrow{ X } [L] (g) + f \bigl( \beta (g) \bigr) .
\end{equation*} 
If the discrete Euler--Lagrange equations $ \overleftarrow{ X } [L] (g
_k ) = \overrightarrow{ X } [L] ( g _{ k + 1 } ) $ hold, then
\begin{align*}
  F _X \bigl( \mathrm{d} L ( g _k ) \bigr) &= \overleftarrow{ X }
  [L] (g _k ) + f \bigl( \beta  (g _k ) \bigr) \\
  &= \overrightarrow{ X } [L] ( g _{ k + 1 } ) + f \bigl( \alpha ( g
  _{ k + 1 } ) \bigr) \\
  &= F _X \bigl( \mathrm{d} L ( g _{ k + 1 } ) \bigr),
\end{align*} 
where the second line uses only the discrete Euler--Lagrange equations
and the composability of $ g _k $ and $ g _{ k + 1 } $. Thus, $ F _X $
is indeed a conserved quantity. (This is precisely the argument
appearing in \citet{MaMaMa2006} for the unconstrained case.)

\section{Variational principles}
\label{sec:variational}

Discrete constrained Lagrangian systems are a generalization of
variational integrators.  However, up to this point, we have only
treated the discrete Lagrangian $ L \colon N \subset G \rightarrow
\mathbb{R} $ as a generating function, and have not yet made any
mention of its variational interpretation.  In this section, we show
that, under certain regularity conditions, the Lagrangian dynamics of
$ \left( G, N, L \right) $ are precisely the critical points of a
discrete constrained action sum.

\subsection{Admissible trajectories and variations}
Given a fixed element $ g \in G $, a trajectory $ g _1, \ldots, g _n
\in G $ is said to be \emph{admissible} if $ g _k \in N $ for $ k = 1,
\ldots, n $, and if the elements are composable with product $ g _1
\cdots g _n = g $.  Let $ \mathcal{C} _{ g, N } ^n $ denote the set of
admissible trajectories of length $n$.  Without loss of generality, we
can take $ n = 2 $, so that the space of admissible trajectories is
\begin{equation*}
  \mathcal{C} ^2 _{ g, N } = \left\{ \left( g _1, g _2 \right) \in G
    _2 \cap \left( N \times N \right) \;\middle\vert\; g _1 g _2 = g
  \right\} = \left( m \rvert _{ G _2 \cap \left( N \times N \right) }
  \right) ^{-1} (g) .
\end{equation*}
Finally, assume that $ G _2 \cap \left( N \times N \right) $ is a
submanifold of $ N \times N $, and that the map $ m \rvert _{ G _2
  \cap \left( N \times N \right) } \colon G _2 \cap \left( N \times N
\right) \rightarrow G $ has constant rank in an open neighborhood of $
\mathcal{C} ^2 _{ g, N } $.  This is sufficient to ensure (by the
subimmersion theorem, cf.~\citet{AbMaRa1988}) that $ \mathcal{C} ^2 _{
  g, N } $ is itself a submanifold of $ G _2 \cap \left( N \times N
\right) $.

Now, a ``variation'' of an admissible trajectory $ \left( g _1, g _2
\right) \in \mathcal{C} ^2 _{ g, N } $ is an element of the tangent
space
\begin{multline*}
  T _{ \left( g _1, g _2 \right) } \mathcal{C} ^2 _{ g, N } = \bigl\{
  \left( v _{ g _1 }, v _{ g _2 } \right) \in T _{ g _1 } N \times T
  _{ g _2 } N \;\bigm\vert\; T _{ g _1 } \beta \left( v _{ g _1 }
  \right) = T _{ g _2 } \alpha \left( v _{ g _2 } \right) \\
  \text{and } T _{ \left( g _1 , g _2 \right) } m \left( v _{ g _1 },
    v _{ g _2 } \right) = 0 \bigr\} .
\end{multline*}
In fact, letting $ q = \beta \left( g _1 \right) = \alpha \left( g _2
\right) $, this tangent space is isomorphic to a vector subspace $
\left( A _q G \right) _{ \left( N, g _1, g _2 \right) } \subset A _q G
$, defined by
\begin{equation*}
  \left( A _q G \right) _{ \left( N, g _1, g _2 \right) } = \left\{  v
    \in A _q G \;\middle\vert\; \overleftarrow{ v } \left( g _1
    \right) \in T _{ g _1 } N \text{ and } \overrightarrow{ v } \left(
    g _2 \right) \in T _{ g _2 } N \right\} ,
\end{equation*}
where we recall that $ \overleftarrow{ v } \left( g _1 \right) = T _{
  \epsilon (q) } \ell _{ g _1 } (v) $ and $ \overrightarrow{ v }
\left( g _2 \right) = -T _{ \epsilon (q) } \left( r _{ g _2 } \circ i
\right) (v) $.  The following theorem establishes this isomorphism
explicitly.

\begin{theorem}
  \label{thm:variationalIsomorphism}
  The linear map $ \varphi _{ \left( N, g _1, g _2 \right) } \colon
  \left( A _q G \right) _{ \left( N, g _1 , g _2 \right) } \rightarrow
  T _{ \left( g _1, g _2 \right) } \mathcal{C} ^2 _{ g, N } $, defined
  by $ v \mapsto \left( \overleftarrow{ v } \left( g _1 \right) , -
    \overrightarrow{ v } \left( g _2 \right) \right) $, is an
  isomorphism.
\end{theorem}

\begin{proof}
  The proof proceeds in three stages: first, showing that $ \varphi _{
    \left( N, g _1, g _2 \right) } $ is well-defined; second, that it
  is injective; and finally, that it is surjective.

  \emph{$ \varphi _{ \left( N, g _1, g _2 \right) } $ is well-defined:}
  Given any $ v \in \left( A _q G \right) _{ \left( N, g _1 , g _2
    \right) } $, we wish to show that $ \varphi _{ \left( N, g _1, g _2
    \right) } (v) \in T _{ \left( g _1, g _2 \right) } \mathcal{C} ^2
  _{ g, N } $.  By definition of the space $ \left( A _q G \right) _{
    \left( N, g _1 , g _2 \right) } $, we clearly have $ \left(
    \overleftarrow{ v } \left( g _1 \right) , - \overrightarrow{ v }
    \left( g _2 \right) \right) \in T _{ g _1 } N \times T _{ g _2 } N
  $.  Moreover,
  \begin{equation*}
    T \beta \left( \overleftarrow{ v } \left( g _1 \right) \right) =
    \rho (v) = T \alpha \left( - \overrightarrow{ v } \left( g _2
      \right) \right) ,
  \end{equation*}
  so all that remains is to show that $ T m \left( \overleftarrow{ v }
    \left( g _1 \right) , - \overrightarrow{ v } \left( g _2 \right)
  \right) = 0 $.  Using the results from \autoref{app:bisections}, let
  us suppose that $ B _1 $ and $ B _2 $ are local bisections of $G$
  such that $ g _1 \in B _1 $ and $ g _2 \in B _2 $, so $ \left( B _1
  \right) _\beta (q) = g _1 $ and $ \left( B _2 \right) _\alpha (q) =
  g _2 $.  Then, applying \autoref{thm:tangentMultiplication} and the
  inversion identity \eqref{eqn:tangentInversion}, we obtain
  \begin{align*}
    T m \left( \overleftarrow{ v } \left( g _1 \right) , -
      \overrightarrow{ v } \left( g _2 \right) \right) &= Tm \left( T
      \ell _{ B _1 } (v) , T \left( r _{ B _2 } \circ i \right) (v)
    \right) \\
    &= T \left( r _{ B _2 } \circ \ell _{ B _1 } \right) (v) + T
    \left( \ell _{ B _1 } \circ r _{ B _2 } \circ i \right) (v)
    \\
    &\qquad - T \left( \ell _{ B _1 } \circ r _{ B _2 } \circ \epsilon
      \circ \beta \right) (v) \\
    &= T \left( r _{ B _2 } \circ \ell _{ B _1 } \right) (v) - T
    \left( \ell _{ B _1 } \circ r _{ B _2 } \right) (v)
  \end{align*}
  which vanishes since $ r _{ B _2 } \circ \ell _{ B _1 } = \ell _{ B
    _1 } \circ r _{ B _2 } $.  Therefore, we have shown that $ \varphi _{
    \left( N, g _1, g _2 \right) } (v) \in T _{ \left( g _1, g _2
    \right) } \mathcal{C} ^2 _{ g, N } $, as claimed.

  \emph{$ \varphi _{ \left( N, g _1, g _2 \right) } $ is injective:} This
  follows immediately from the fact that the linear map $ T _{
    \epsilon (q) } \ell _{ g _1 } \colon A _q G \rightarrow T _{ g _1
  } G $, $ v \mapsto \overleftarrow{ v } \left( g _1 \right) $, is
  injective.

  \emph{$ \varphi _{ \left( N, g _1, g _2 \right) } $ is surjective:}
  Suppose that $ \left( v _{ g _1 }, v _{ g _2 } \right) \in T _{
    \left( g _1, g _2 \right) } \mathcal{C} ^2 _{ g, N } $.  Then we
  can take some curve in $ \mathcal{C} ^2 _{ g, N } $,
  \begin{equation*}
    c \colon \left( - \delta, \delta \right) \subset \mathbb{R}
    \rightarrow \mathcal{C} ^2 _{ g, N } , \quad t \mapsto c (t) =
    \left( c _1 (t) , c _2 (t) \right) ,
  \end{equation*}
  such that $ c (0) = \left( c _1 (0) , c _2 (0) \right) = \left( g
    _1, g _2 \right) $ and $ \dot{ c } (0) = \left( v _{ g _1 } , v _{
      g _2 } \right) $.

  Now, since $ \alpha \left( c _1 (t) \right) = \alpha \left( g _1
  \right) $ for all $t$, we can write
  \begin{equation*}
    c _1 (t) = g _1 h (t), 
  \end{equation*}
  where $ h (0) = \epsilon \left( \beta \left( g _1 \right) \right) $
  and $ h (t) = g _1 ^{-1} c _1 (t) \in \alpha ^{-1} \left( \beta
    \left( g _1 \right) \right) $.  Therefore, this defines a Lie
  algebroid element $ v = \dot{ h } (0) \in A _q G $, and
  \begin{equation*}
    v _{ g _1 } = \dot{ c } _1 (0) = T \ell _{ g _1 } \bigl( \dot{ h }
      (0) \bigr) = \overleftarrow{ v } \left( g _1 \right) .
  \end{equation*}
  On the other hand, since $ c _1 (t) c _2 (t) = g _1 g _2 $ for all
  $t$, it follows that $ c _2 (t) = h (t) ^{-1} g _2 $.  Thus,
  \begin{equation*}
    v _{ g _2 } = \dot{c} _2 (0) = T \left( r _{ g _2 } \circ i
    \right) \bigl( \dot{h} (0) \bigr) = - \overrightarrow{ v } \left(
      g _2 \right) .
  \end{equation*}
  Therefore, we have shown that $ v \in \left( A _q G \right) _{
    \left( N, g _1 , g _2 \right) } $ and that $ \left( v _{ g _1 } ,
    v _{ g _2 } \right) = \varphi _{ \left( N, g _1, g _2 \right) } (v)
  $, so $ \varphi _{ \left( N, g _1, g _2 \right) } $ is surjective.
  This completes the proof.
\end{proof}

\subsection{The discrete constrained action principle} Now that we
have characterized the admissible trajectories of length two, and
their variations, we introduce the \emph{discrete constrained action
  sum}
\begin{equation*}
  S \left( G, N, L \right) ^2 \colon \mathcal{C} ^2 _{ g, N }
  \rightarrow \mathbb{R}  , \quad \left( g _1, g _2 \right) \mapsto L
  \left( g _1 \right) + L \left( g _2 \right) .
\end{equation*}
A \emph{critical point} of $ S \left( G, N, L \right) ^2 $ is a
trajectory $ \left( g _1, g _2 \right) \in \mathcal{C} ^2 _{ g, N } $
such that $ \mathrm{d} S \left( G, N, L \right) ^2 \left( g _1, g _2
\right) = 0 $, i.e.,
\begin{equation*}
  0 = \bigl\langle \mathrm{d} S \left( G, N, L \right) ^2 , \left(
    \overleftarrow{ v } \left( g _1 \right) , - \overrightarrow{ v }
    \left( g _2 \right) \right) \bigr\rangle = \left\langle \mathrm{d}
    L ,\overleftarrow{ v } \left( g _1 \right) \right\rangle -
  \left\langle \mathrm{d} L, \overrightarrow{ v } \left( g _2 \right)
  \right\rangle ,
\end{equation*}
for all $ v \in \left( A _q G \right) _{ \left( N, g _1, g _2
  \right) } $, where as before $ q = \beta \left( g _1 \right) =
\alpha \left( g _2 \right) $.

To express these discrete constrained Euler--Lagrange equations in
terms of arbitrary sections $ X \in \Gamma \left( A G \right) $, as
with the earlier results, suppose now that $ \widehat{L} $ is an
extension to a neighborhood of $N \subset G$, so that $ L =
\widehat{L} \rvert _N $.  Therefore, it follows from the above that
\begin{multline}
  \label{eqn:annihilator}
  \mathrm{d} \left( \widehat{L} \circ \ell _{ g _1 } + \widehat{L}
    \circ r _{ g _2 } \circ i \right) \left( \epsilon (q) \right) \\
  = T ^\ast _{ \epsilon (q) } \ell _{ g _1 } \left( \mathrm{d}
    \widehat{L} \left( g _1 \right) \right) + T ^\ast _{ \epsilon (q)
  } \left( r _{ g _2 } \circ i \right) \left( \mathrm{d} \widehat{L}
    \left( g _2 \right) \right) \in \left( A _q G \right) _{ \left( N,
      g _1, g _2 \right) } ^0 ,
\end{multline}
where $ T ^\ast _{ \epsilon (q) } \ell _{ g _1 } $ and $ T ^\ast _{
  \epsilon (q) } \left( r _{ g _2 } \circ i \right) $ denote the
cotangent lift (i.e., the adjoint to the tangent map) of $ \ell _{ g
  _1 } $ and $ r _{ g _2 } \circ i $, respectively, and where $ \left(
  A _q G \right) _{ \left( N, g _1, g _2 \right) } ^0 \subset T ^\ast
_{\epsilon(q)} G $ is the annihilator of $ \left( A _q G \right) _{
  \left( N, g _1, g _2 \right) } $.  Now, by using the definition of $
\left( A _q G \right) _{ \left( N, g _1, g _2 \right) } $, we can
write its annihilator as
\begin{multline*}
    \left( A _q G \right) _{ \left( N, g _1, g _2 \right) } ^0 \\
    \begin{aligned}
      &= \left[ T \ell _{ g _1 } ^{-1} \left( T _{ g _1 } N \cap \ker
          T _{ g _1 } \alpha \right) \cap T \left( r _{ g _2 } \circ i
        \right) ^{-1} \left( T _{ g _2 } N \cap \ker T _{ g _2 } \beta
        \right) \right]
      ^0 \\
      &= \left[ T \ell _{ g _1 } ^{-1} \left( T _{ g _1 } N \cap \ker
          T _{ g _1 } \alpha \right) \right] ^0 + \left[ T \left( r _{
            g _2 } \circ i \right) ^{-1} \left( T _{ g _2 } N \cap
          \ker T _{ g _2 } \beta \right) \right] ^0 \\
      &= T ^\ast _{ \epsilon (q) } \ell _{ g _1 } \left( \left( T _{ g
            _1 } N \cap \ker T _{ g _1 } \alpha \right) ^0 \right) + T
      ^\ast _{ \epsilon (q) } \left( r _{ g _2 } \circ i \right)
      \left( \left( T _{ g _2 } N \cap  \ker T _{ g _2 } \beta \right)
        ^0 \right)      .
    \end{aligned}
\end{multline*}
Therefore, from \eqref{eqn:annihilator}, we obtain that
\begin{equation*}
  T ^\ast _{ \epsilon (q) } \ell _{ g _1 } \left( \mathrm{d}
    \widehat{L} \left( g _1 \right) + \widetilde{ \Lambda } _1 \right)
  = -   T ^\ast _{ \epsilon (q) } \left( r _{ g _2 } \circ i \right)
  \left( \mathrm{d} \widehat{L} \left( g _2 \right) + \widetilde{
      \Lambda } _2 \right),
\end{equation*}
where we have denoted $ \widetilde{ \Lambda } _1 \in \left( T _{ g _1
  } N \cap \ker T _{ g _1 } \alpha \right) ^0 = \left( T _{ g _1 } N
\right) ^0 + \left( \ker T _{ g _1 } \alpha \right) ^0 $ and $
\widetilde{ \Lambda } _2 \in \left( T _{ g _2 } N \cap \ker T _{ g _2
  } \beta \right) ^0 = \left( T _{ g _2 } N \right) ^0 + \left( \ker T
  _{ g _2 } \beta \right) ^0 $.  Consequently, we can write
\begin{equation*}
  \widetilde{ \Lambda } _1 = \Lambda _1 + \mathring{ \Lambda } _1 ,
  \qquad \widetilde{ \Lambda } _2 = \Lambda _2 + \mathring{ \Lambda }
  _2 ,
\end{equation*}
where $ \Lambda _1 \in \left( T _{ g _1 } N \right) ^0 = \nu ^\ast _{
  g _1 } N $, $ \mathring{ \Lambda } _1 \in \left( \ker T _{ g _1 }
  \alpha \right) ^0 $, $ \Lambda _2 \in \left( T _{ g _2 } N \right)
^0 = \nu ^\ast _{ g _2 } N $, and $ \mathring{ \Lambda } _2 \in \left(
  \ker T _{ g _2 } \beta \right) ^0 $.  Finally, since $
\mathring{\Lambda} _1 $ annihilates $\alpha$-vertical tangent vectors,
in particular it annihilates any left-invariant vector field evaluated
at $ g _1 $; likewise, $ \mathring{\Lambda } _2 $ annihilates any
right-invariant vector field evaluated at $ g _2 $. Therefore, it
follows that for any $ X \in \Gamma \left( A G \right) $, we have
\begin{equation*}
  \overleftarrow{ X } [\widehat{L}] \left( g _1 \right)  +
  \bigl\langle \Lambda _1 , \overleftarrow{ X } \left( g _1 \right)
  \bigr\rangle  =   \overrightarrow{ X } [\widehat{L}] \left( g _2
  \right)  + \bigl\langle \Lambda _2 , \overrightarrow{  X } \left( g
    _2 \right) \bigr\rangle ,
\end{equation*}
which is precisely the equation \eqref{eqn:delLambda} that we derived
earlier, in the case $ n = 2 $.

To generalize this argument for any $ n \geq 2 $, one can proceed in a
similar manner.  Given a fixed $ g \in G $, suppose that the space of
composable sequences of length $n$ in $N$,
\begin{equation*}
  N _2 ^n = \left\{ \left( g _1, \ldots, g _n \right) \in N
    \;\middle\vert\; \left( g _k , g _{ k + 1 } \right) \in G _2
    \text{ for } k = 1, \ldots, n - 1 \right\} ,
\end{equation*}
is a submanifold of $ N ^n $, and that the multiplication map
\begin{equation*}
  m ^n \rvert _{ N ^n _2  } \colon N _2 ^n \rightarrow G, \quad \left(
    g _1 , \ldots, g _n \right) \mapsto g _1 \cdots g _n ,
\end{equation*}
has constant rank in an open neighborhood of $ \mathcal{C} ^n _{ g, N
} = \left( m ^n \rvert _{ N ^n _2 } \right) ^{-1} (g) $.  Then, as
before, it follows that $ \mathcal{C} ^n _{ g, N } $ is a submanifold
of $ N _2 ^n $.  The discrete constrained action sum is then
\begin{equation*}
  S \left( G, N, L \right) ^n \colon \mathcal{C} ^n _{ g, N }
  \rightarrow \mathbb{R}  , \quad \left( g _1, \ldots, g _n \right)
  \mapsto \sum _{ k = 1 } ^n L \left( g _k \right) ,
\end{equation*}
and one may prove that when $ \left( g _1, \ldots, g _n \right) \in
\mathcal{C} ^n _{ g, N } $ is a critical point of $ S \left( G, N, L
\right) ^n $, the equations
\begin{equation*}
  \overleftarrow{ X } [\widehat{L}] \left( g _k \right)  +
  \bigl\langle \Lambda _k , \overleftarrow{ X } \left( g _k \right)
  \bigr\rangle  =   \overrightarrow{ X } [\widehat{L}] \left( g _{ k +
      1 } \right)  + \bigl\langle \Lambda _{ k + 1 } ,
  \overrightarrow{  X } \left( g _{ k + 1 } \right) \bigr\rangle ,
\end{equation*}
hold for all sections $ X \in \Gamma \left( A G \right) $ and $ k = 1,
\ldots, n - 1 $, where $ \Lambda _k \in \nu ^\ast _{ g _k } N $ for $
k = 1, \ldots, n $.  Again, this agrees precisely with the earlier
equation \eqref{eqn:delLambda}.

\begin{remark}
  As in \autoref{rmk:subgroupoid}, we note that this discrete
  constrained action principle is different from that used in
  nonholonomic mechanics, where the dynamics are generally not
  symplectic or Poisson. Rather, this more closely resembles the
  vakonomic approach used in optimal control theory, among other
  applications.
\end{remark}

\section{Examples}
\label{sec:examples}

\subsection{Constrained mechanics and optimal control on Lie groups} 
\label{sec:optimalControl}
Let $G$ be a Lie group, whose identity element is denoted by $e \in G
$, and let $\mathfrak{g}$ be the Lie algebra of $G$.  Suppose that a
submanifold $ N \subset G $ is given by the vanishing of the
constraint functions $ \phi ^a $, for $ a \in A $, defined on a
neighborhood of $N$.  Let $ L \colon N \rightarrow \mathbb{R} $ be a
discrete Lagrangian, and let $ \widehat{L} $ be an extension of $L$ to
a neighborhood of $N$, so that $ L = \widehat{L} \rvert _N $.  Then
the discrete dynamics are given by the constraint equations
\begin{equation*}
  \phi ^a \left( g _k \right) = 0 \text{ for all } a \in A, \quad k =
  1, \ldots, n ,
\end{equation*}
together with the equations
\begin{equation}
  \label{eqn:lieGroup}
  \overleftarrow{ \xi } \bigl[ \widehat{L} + \left( \lambda _k \right)
  _a \phi ^a \bigr]  \left( g _k \right)  = \overrightarrow{ \xi }
  \bigl[ \widehat{L} + \left( \lambda _{ k + 1 } \right) _a \phi ^a
  \bigr] \left( g _{ k + 1 } \right) , \quad k = 1, \ldots, n - 1 ,
\end{equation}
for every $ \xi \in \mathfrak{g} $.

Now, in addition to the usual definition $ \overleftarrow{ \xi } (g) =
T \ell _g (\xi) $, we can use the tangent inversion identity
\eqref{eqn:tangentInversion} to see that $ T i (\xi) = - \xi + T
\left( \epsilon \circ \beta \right) (\xi) = - \xi $; therefore, by the
chain rule, we simply have $ \overrightarrow{ \xi } (g) = - T \left( r
  _g \circ i \right) (\xi) = T r _g (\xi) $.  Hence,
\eqref{eqn:lieGroup} can be rewritten as
\begin{equation*}
  \ell _{ g _k } ^\ast \mathrm{d} \left( \widehat{L} + \left( \lambda _k
    \right) _a \phi ^a \right) (e) = r _{ g _{ k + 1 } } ^\ast \mathrm{d}
  \left( \widehat{L} + \left( \lambda _{k+1} \right) _a \phi ^a
  \right) (e) , \quad k = 1 , \ldots, n - 1 .
\end{equation*}
If we define $ \mu _k = r _{ g _k } ^\ast \mathrm{d} \widehat{L} (e)
\in \mathfrak{g} ^\ast $ and $ \Phi ^a _k = r _{ g _k } ^\ast
\mathrm{d} \phi ^a (e) \in \mathfrak{g} ^\ast $ for each $k$, then
this is equivalent to
\begin{equation*}
  \mu _{ k + 1 } + \left( \lambda _{ k + 1 } \right) _a \Phi ^a _{k+1} =
  \operatorname{Ad} ^\ast _{ g _k } \bigl( \mu _k + \left( \lambda _k
    \right) _a \Phi ^a _k \bigr), \quad k = 1, \ldots, n - 1 ,
\end{equation*}
where $ \operatorname{Ad} ^\ast \colon G \times \mathfrak{g}^{\ast}
\rightarrow \mathfrak{g}^{\ast} $ is the coadjoint action of $G$ on
$\mathfrak{g}^{\ast}$.  These equations will be called the
\emph{discrete constrained Lie--Poisson equations} for this system.

We now show how this framework can be used to discretize a general
family of optimal control problems on Lie groups.  Again, let $G$ be a
Lie group with Lie algebra $\mathfrak{g}$, and define some
\emph{control submanifold} $ C \subset \mathfrak{g} $.  A curve $ \xi
\colon [0,1] \rightarrow C $ is said to be a \emph{control curve} for
a trajectory $ g \colon [0,1] \rightarrow G $ if it satisfies $ \dot{
  g } (t) = T \ell _{ g (t) } \left( \xi (t) \right) $ for all $t$;
when $G$ is a matrix Lie group, we can simply write this as $ \dot{g}
(t) = g (t) \xi (t) $.  The continuous optimal control problem is
defined as follows (cf.~\citet{KoMa1997}):
\begin{quote}
  Given an initial configuration $ g _0 \in G $, a final configuration
  $ g _1 \in G $, and a function $ \mathfrak{l} \colon \mathfrak{g}
  \rightarrow \mathbb{R} $, find the control curve $ \xi \colon [0,1]
  \rightarrow C $ for a path $ g \colon [0,1] \rightarrow G
  $, satisfying $ g (0) = g _0 $ and $ g (1) = g _1 $, such that $
  \int _0 ^1 \mathfrak{l} \left( \xi (t) \right) \,\mathrm{d}t $ is
  minimized.
\end{quote}
In other words, we wish to find the most efficient (lowest-cost) way
to ``steer'' the system on a path from $ g _0 $ to $ g _1 $ by
applying controls in $C$.  (In optimal control problems, it is
generally assumed that the system is \emph{controllable}, i.e., that
there exists a control curve for every pair of initial and final
configurations.  As in \citet{KoMa1997}, we omit discussion of
controllability, since we are primarily interested in the necessary
variational conditions for optimality.)

To derive the associated discrete system, it is usually necessary to
introduce an analytic local diffeomorphism $ \tau \colon \mathfrak{g}
\rightarrow G $, which maps a neighborhood of $ 0 \in \mathfrak{g}$ to
a neighborhood of $ e \in G $, where $e$ denotes the identity element
of $G$.  As a consequence, it is possible to deduce that $ \tau (\xi)
\tau \left( - \xi \right) = e $.  This map can then be used to relate
the continuous and discrete settings (cf.~\citet{BoMa2009}).  There
are several choices, in the literature, for the map $\tau$.  For
instance, one may take $\tau$ to be the exponential map, defined by
the time-$1$ flow $ \exp (\xi) = \gamma (1) $, where $ \gamma \colon
[0,1] \rightarrow G $ is the integral curve of the vector field $
\overleftarrow{ \xi } \in \mathfrak{X} (G) $ with initial condition $
\gamma (0) = e $.  Alternatively, for quadratic matrix Lie groups
(e.g., $\operatorname{SO(3)}$, $ \operatorname{SE(2)} $, $
\operatorname{SE(3)} $), it is typical to use the Cayley map $
\operatorname{cay} \colon \mathfrak{g} \rightarrow G $, defined by $
\operatorname{cay} (\xi) = \left( I - \xi / 2 \right) ^{-1} \left( I +
  \xi / 2 \right) $, where $ I = e $ denotes the identity matrix.

Now, suppose that we start with a continuous optimal control problem,
specified by a Lagrangian $ \mathfrak{l} \colon \mathfrak{g}
\rightarrow \mathbb{R} $ and a control submanifold $C$ determined by
the vanishing of the independent constraints $ \Psi ^a \colon
\mathfrak{g} \rightarrow \mathbb{R} $, $ a \in A $.  Then, given a
choice of the map $ \tau \colon \mathfrak{g} \rightarrow G $, we may
construct the discrete Lagrangian $ \widehat{L} \colon G \rightarrow
\mathbb{R} $ and constraint functions $ \phi ^a \colon G \rightarrow
\mathbb{R} $, $ a \in A $, by
\begin{equation*}
  \widehat{L} (g) = h \mathfrak{l} \left( \frac{ \tau ^{-1} (g) }{ h }
  \right) , \qquad \phi ^a (g) = h \Psi ^a \left( \frac{ \tau ^{-1} (g) }{
      h } \right) ,
\end{equation*}
where $h$ is the time step size of the discretization.

Therefore, we can finally rewrite Equation~\eqref{eqn:lieGroup} as
\begin{multline*}
  \left( \tau ^{-1} \circ \ell _{ g _k } \right) ^\ast \left(
    \mathrm{d} \mathfrak{l} \left( \xi _k \right) + \left( \lambda _k
    \right) _a \mathrm{d} \Psi ^a \left( \xi _k \right) \right) \\
  = \left( \tau ^{-1} \circ r _{ g _{ k + 1 } }\right) ^\ast \left(
    \mathrm{d} \mathfrak{l} \left( \xi _{ k + 1 } \right) + \left(
      \lambda _{ k + 1 } \right) _a \mathrm{d} \Psi ^a \left( \xi _{ k
        + 1 } \right) \right), \quad k = 1 , \ldots, n - 1 ,
\end{multline*} 
where $ \xi _k = \tau ^{-1} \left( g _k \right) / h $.  In other
words, denoting $ \mathrm{d} \tau _\xi \colon \mathfrak{g} \rightarrow
\mathfrak{g} $ to be the left-trivialized tangent map of $\tau$,
defined by $ T _{ \xi } \tau = T _e \ell _{ \tau (\xi) } \circ
\mathrm{d} \tau _\xi $, the previous equation can be rewritten as
\begin{multline*}
  \bigl( \mathrm{d} \tau _{ h \xi _k } ^{-1} \bigr) ^\ast \left(
    \mathrm{d} \mathfrak{l} \left( \xi _k \right) + \left( \lambda _k
    \right) _a \mathrm{d} \Psi ^a \left( \xi _k \right) \right) \\
  = \bigl( \mathrm{d} \tau _{ - h \xi _{ k + 1 } } ^{-1} \bigr) ^\ast
  \left( \mathrm{d} \mathfrak{l} \left( \xi _{ k + 1 } \right) +
    \left( \lambda _{ k + 1 } \right) _a \mathrm{d} \Psi ^a \left( \xi
      _{ k + 1 } \right) \right) , \quad k = 1, \ldots, n - 1 ,
\end{multline*}
since we have $ \mathrm{d} \tau _\xi ^{-1} = T _{ \tau (\xi) } \tau
^{-1} \circ T _e \ell _{ \tau (\xi) } $.

\subsection{A discrete plate-ball system}

Consider, now, the optimal control problem for the classical example
of a homogeneous ball, rolling (without slipping) on a rotating plate.
The configuration space of this system is $ Q = \mathbb{R}^2 \times
\operatorname{SO(3)} $, parameterized by $ q = \left( x, y, g \right)
$.  Let $ \left( \omega _x, \omega _y , \omega _z \right) \in
\mathbb{R}^3 \cong \mathfrak{so}(3) $ be the angular velocity vector
of the ball with respect to an inertial frame, and let $r$ be its
radius.  Suppose that the plane rotates with constant angular velocity
$\Omega$, about the axis perpendicular to the plane through the
origin.  Then, we may formulate the following optimal control problem,
which is called the \emph{plate-ball problem}:
\begin{quote}
  Given initial and final configurations $ q _0, q _1 \in \mathbb{R}^2
  \times \operatorname{SO(3)} $, find the optimal control curves $
  \left( x (t) , y (t) \right) \in \mathbb{R}^2 $, $ t \in [0,1] $,
  that steer the system from $ q _0 $ to $ q _1 $, such that the cost
  function
  \begin{equation*}
    \int _0 ^1 \frac{1}{2} \left[ \left( \dot{x} (t) \right) ^2 +
      \left( \dot{y} (t) \right) ^2 \right] \,\mathrm{d}t 
  \end{equation*}
  is minimized, subject to the (non-integrable)
  rolling-without-slipping constraints
  \begin{equation*}
    \dot{y} + r \omega _x = \Omega x, \qquad \dot{x} - r \omega _y = -
    \Omega y , \qquad \omega _z = c , 
  \end{equation*}
  where $c$ is a constant.
\end{quote}
The continuous equations of motion for this problem are carefully
studied in \citet{KoMa1997,IgMaMaSo2008}; we now study the geometric
discretization of this system and its dynamics.

The Lie groupoid that arises for the plate-ball problem is $
\mathbb{R} ^2 \times \mathbb{R}^2 \times \operatorname{SO(3)}
\rightrightarrows \mathbb{R}^2 $, where the structure functions are
given by
\begin{align*}
  \alpha \left( x _0 , y _0, x _1, y _1 , g _1 \right) &= \left( x _0,
    y _0 \right) ,\\
  \beta \left( x _0, y _0, x _1, y _1, g _1
  \right) &= \left( x _1 , y _1 \right) ,\\
  m \left( \left( x _0 , y _0, x _1, y _1 , g _1 \right) , \left( x
      _1, y _1 , x _2, y _2, g _2 \right) \right) &= \left( x _0, y
    _0,
    x _2, y _2, g _1 g _2 \right) ,\\
  i \left( x _0, y _0, x _1, y _1, g _1 \right) &= \left( x _1, y _1 ,
    x _0, y _0, g _1 ^{-1} \right) ,\\
  \epsilon \left( x, y \right) &= \left( x , y, x, y, I \right) .
\end{align*}
Now, denote by $ \left\{ E _1, E _2, E _3 \right\} $ the standard
basis of the Lie algebra $ \mathfrak{so}(3) $,
\begin{equation*}
  E _1 =
  \begin{pmatrix*}[r]
    0&0&0\\
    0&0&-1\\
    0&1&0
  \end{pmatrix*},
  \qquad 
  E _2 =
  \begin{pmatrix*}[r]
    0&0&1\\
    0&0&0\\
    -1&0&0
  \end{pmatrix*},
  \qquad
  E _3 = 
  \begin{pmatrix*}[r]
    0&-1&0\\
    1&0&0\\
    0&0&0
  \end{pmatrix*} .
\end{equation*}
This defines the Lie algebra isomorphism $ \widehat{\phantom{n}}
\colon \mathbb{R}^3 \rightarrow \mathfrak{so}(3) $ (sometimes called
the ``hat map'' \citep{AbMaRa1988}), which takes
\begin{equation*}
  \omega = \left( \omega _x, \omega _y , \omega _z \right) \mapsto
  \widehat{\omega} = \omega _x E _1 + \omega _y E _2 + \omega _z E _3
  =
  \begin{pmatrix}
    0&-\omega _z & \omega _y \\
    \omega _z & 0 & -\omega _x \\
    - \omega _y & \omega _x & 0 
  \end{pmatrix} .
\end{equation*}
As a discretization procedure for the Lie group $ \operatorname{SO(3)}
$, it is typical to use the matrix logarithm, denoted $ \log \colon
\operatorname{SO(3)} \rightarrow \mathfrak{so}(3) $, which is the
(local) inverse of the matrix exponential map $ \exp \colon
\mathfrak{so}(3) \rightarrow \operatorname{SO(3)} $.  For simplicity,
we may use the approximation $ \exp \left( h \xi \right) \approx I + h
\xi $ to obtain
\begin{equation*}
  \xi = \frac{ \log g }{ h } \approx \frac{ g - I }{ h } .
\end{equation*}
Taking $ \widehat{\omega} = \xi $, we therefore deduce that
\begin{alignat*}{2}
  \omega _x &= - \frac{1}{2} \operatorname{tr} \left( \xi E _1 \right)
  &&\approx - \frac{ 1 }{ 2 h } \operatorname{tr} \left( g E _1
  \right)
  , \\
  \omega _y &= - \frac{1}{2} \operatorname{tr} \left( \xi E _2 \right)
  &&\approx - \frac{ 1 }{ 2 h } \operatorname{tr} \left( g E _2 \right),\\
  \omega _z &= - \frac{1}{2} \operatorname{tr} \left( \xi E _3 \right)
  &&\approx - \frac{ 1 }{ 2 h } \operatorname{tr} \left( g E _3
  \right) ,
\end{alignat*}
since $ \operatorname{tr} \left( I E _i \right) = \operatorname{tr}
\left( E _i \right) = 0 $ for $ i = 1, 2, 3 $.

Finally, we derive the following discretization of the plate-ball
system.  Let $N$ be the submanifold of $ \mathbb{R}^2  \times
\mathbb{R}^2  \times \operatorname{SO(3)} $ determined by the
vanishing of the constraint functions,
\begin{equation}
  \label{eqn:discretePlateBallConstraints}
  \begin{aligned}
    \phi ^1 \left( x _0, y _0, x _1 , y _1 , g _1 \right) &= h \left[
      \frac{ y _1 - y _0 }{ h } - \frac{ r }{ 2 h } \operatorname{tr}
      \left( g _1 E
        _1 \right) - \Omega \frac{ x _1 + x _0 }{ 2 } \right] ,\\
    \phi ^2 \left( x _0, y _0, x _1 , y _1 , g _1 \right) &= h \left[
      \frac{ x _1 - x _0 }{ h } + \frac{ r }{ 2 h } \operatorname{tr}
      \left( g _1 E
        _2 \right) + \Omega \frac{ y _1 + y _0 }{ 2 } \right] ,\\
    \phi ^3 \left( x _0, y _0, x _1 , y _1 , g _1 \right) &= h \left[
      c + \frac{ 1 }{ 2 h } \operatorname{tr} \left( g _1 E _3 \right)
    \right] ,
  \end{aligned}
\end{equation}
and take the discrete Lagrangian to be
\begin{equation*}
  \widehat{L} \left( x _0, y _0, x _1 , y _1 , g _1 \right) =
  \frac{h}{2} \left[ \left( \frac{ x _1 - x _0 }{ h } \right) ^2 +
    \left( \frac{ y _1 - y _0 }{ h } \right) ^2 \right] .
\end{equation*}
Therefore, applying the results of \autoref{sec:lagrangeMultipliers},
the discrete constrained Lagrangian dynamics are given by
\begin{multline*}
  \overleftarrow{ X } \bigl[ \widehat{L} + \left( \lambda _k \right)
  _a \phi ^a \bigr] \left( x _{k-1} , y _{k-1} , x _k , y _k , g _k
  \right) \\
  = \overrightarrow{ X } \bigl[ \widehat{L} + \left( \lambda _{ k + 1
    } \right) _a \phi ^a \bigr] \left( x _k , y _k , x _{ k + 1 } , y
    _{ k + 1 }, g _{ k + 1 } \right) , \quad k = 1, \ldots, n - 1 ,
\end{multline*}
together with the vanishing of the constraint functions
\eqref{eqn:discretePlateBallConstraints}, where $X$ is an arbitrary
section of the vector bundle $ T \mathbb{R}^2 \times \mathfrak{so}(3)
\rightarrow \mathbb{R}^2 $.  A basis of sections of this vector bundle
is given by
\begin{equation*}
  \displaystyle
  \left\{ \bigl( \tfrac{ \partial }{ \partial  x } , 0 \bigr) , \bigl(
      \tfrac{ \partial }{ \partial y } , 0 \bigr) , \left( 0, E _1
    \right) , \left( 0, E _2 \right) , \left( 0 , E _3 \right)
  \right\} .
\end{equation*}
Hence, the discrete dynamics are given by the following system of
equations:
\begin{align*}
  0 &= \frac{ x _{ k + 1 } - 2 x _k + x _{ k - 1 } }{ h ^2 } + \frac{
    \left( \lambda _{ k + 1 } \right) _2 - \left( \lambda _k \right)
    _2 }{ h } + \Omega \frac{ \left( \lambda _{k+1} \right) _1 +
    \left( \lambda _k \right) _1 }{ 2 } ,\\
  0 &= \frac{ y _{ k + 1 } - 2 y _k + y _{ k - 1 } }{ h ^2 } + \frac{
    \left( \lambda _{ k + 1 } \right) _2 - \left( \lambda _k \right)
    _2 }{ h } - \Omega \frac{ \left( \lambda _{k+1} \right) _1 +
    \left( \lambda _k \right) _1 }{ 2 } ,\\
  0 &= - r \left( \lambda _k \right) _1 \operatorname{tr} \left( g _k
    E _1 ^2 \right) + r \left( \lambda _{ k + 1 } \right) _1
  \operatorname{tr} \left( E _1 g _{ k + 1 } E _1 \right) \\
  &\qquad + r \left( \lambda _k \right) _2 \operatorname{tr} \left( g
    _k E _1 E _2 \right) - r \left( \lambda _{ k + 1 } \right) _2
  \operatorname{tr} \left( E _1 g _{ k + 1 } E _2 \right) \\
  &\qquad - \left( \lambda _k \right) _3 \operatorname{tr} \left( g _k
    E _1 E _3 \right) + \left( \lambda _{ k + 1 } \right) _3
  \operatorname{tr} \left( E _1 g _{ k + 1 } E _3 \right) ,\\
  0 &= - r \left( \lambda _k \right) _1 \operatorname{tr} \left( g _k
    E _2 E _1 \right) + r \left( \lambda _{k+1} \right) _1
  \operatorname{tr} \left( E _2 g _{ k + 1 } E _1 \right) \\
  &\qquad + r \left( \lambda _k \right) _2 \operatorname{tr} \left( g
    _k E _2 ^2 \right) - r \left( \lambda _{ k + 1 } \right) _2
  \operatorname{tr} \left( E _2 g _{ k + 1 } E _2 \right) \\
  &\qquad - \left( \lambda _k \right) _3 \operatorname{tr} \left( g _k
    E _2 E _3 \right) + \left( \lambda _{ k + 1 } \right) _3
  \operatorname{tr} \left( E _2 g _{ k + 1 } E _3 \right) ,\\
  0 &= - r \left( \lambda _k \right) _1 \operatorname{tr} \left( g _k
    E _3 E _1 \right) + r \left( \lambda _{ k + 1 } \right) _1
  \operatorname{tr} \left( E _3 g _{ k + 1 } E _1 \right) \\
  &\qquad + r \left( \lambda _k \right) _2 \operatorname{tr} \left( g
    _k E _3 E _2 \right) - r \left( \lambda _{ k + 1 } \right) _2
  \operatorname{tr} \left( E _3 g _{ k + 1 } E _2 \right) \\
  &\qquad - \left( \lambda _k \right) _3 \operatorname{tr} \left( g _k
    E _3 ^2 \right) + \left( \lambda _{ k + 1 } \right) _3
  \operatorname{tr} \left( E _3 g _{ k + 1 } E _3 \right) ,\\
  0 &= \frac{ y _{k+1} - y _k }{ h } - \frac{ r }{ 2 h }
  \operatorname{tr} \left( g _{k+1} E _1 \right) - \Omega \frac{ x
    _{k+1} + x_k }{ 2 } ,\\
  0 &= \frac{ x _{k+1} - x _k }{ h } + \frac{ r }{ 2 h }
  \operatorname{tr} \left( g _{k+1} E _2 \right) + \Omega \frac{ y
    _{k+1} + y _k }{ 2 } ,\\
  0 &= c + \frac{ 1 }{ 2 h } \operatorname{tr} \left( g _{k+1} E _3
  \right) .
\end{align*}
These are eight equations (corresponding to the five basis elements
and three constraints), and each step of the dynamics requires solving
for eight unknowns (the five degrees of freedom $x$, $y$, $ \omega _x
$, $ \omega _y $, $ \omega _z $, plus the three Lagrange multipliers).

\subsection{Time-dependent constrained mechanics}

To treat time-dependent discrete mechanics on a Lie groupoid $ G
\rightrightarrows Q $, we construct the new Lie groupoid $ G _{
  \mathbb{R} } = \mathbb{R} \times \mathbb{R} \times G
\rightrightarrows \mathbb{R} \times Q $, with structure functions
defined by
\begin{align*}
  \alpha _{ \mathbb{R} } \left( t _0, t _1 , g _1 \right) &= \left( t
    _0, \alpha \left( g _1 \right) \right) ,\\
  \beta _{ \mathbb{R}  } \left( t _0, t _1 , g _1 \right) &= \left( t
    _1, \beta \left( g _1 \right) \right) ,\\
  m _{ \mathbb{R}  } \left( \left( t _0, t _1, g _1 \right) , \left( t
      _1, t _2 , g _2 \right) \right) &= \left( t _0, t _2, g _1 g _2
  \right) ,\\
  i _{ \mathbb{R}  } \left( t _0, t _1, g _1 \right) &= \left( t _1, t
    _0, g _1 ^{-1} \right) ,\\
  \epsilon _{ \mathbb{R}  } \left( t, q \right) &= \left( t, t,
    \epsilon (q) \right) .
\end{align*}
The associated Lie algebroid is naturally identified with $ T
\mathbb{R} \times A G \rightarrow \mathbb{R} \times Q $.  Hence, the
sections of this Lie algebroid are spanned by $ \left( \frac{ \partial
  }{ \partial t } , 0 \right) $ and elements having the form $ \left(
  0, X \right) $, where $ X \in \Gamma \left( A G \right) $.
Therefore, we can express the left- and right-invariant vector fields
as follows:
\begin{alignat*}{2}
  \overleftarrow{ \left( \tfrac{ \partial }{ \partial t } , 0 \right)
  } \left( t _0, t _1 , g _1 \right) &= \left( - \tfrac{ \partial
    }{ \partial t } \bigr\rvert _{ t = t _0 } , 0 _{ t _1 } , 0 _{ g
      _1 } \right) &&\eqqcolon - \tfrac{ \partial }{ \partial t }
  \bigr\rvert _{ \left( t _0, t _1, g _1 \right)  } ,\\
  \overrightarrow{ \left( \tfrac{ \partial }{ \partial t } , 0 \right)
  } \left( t _0, t _1 , g _1 \right) &= \left( 0 _{ t _0 } ,
    \tfrac{ \partial }{ \partial t } \bigr\rvert _{ t = t _1 } , 0 _{
      g _1 } \right) &&\eqqcolon \tfrac{ \partial }{ \partial t
    ^\prime } \bigr\rvert _{ \left( t _0, t _1, g _1 \right) } ,\\
  \overleftarrow{ \left( 0, X \right) } \left( t _0, t _1 , g _1
  \right) &= \left( 0 _{ t _0 } , 0 _{ t _1 } , \overleftarrow{ X }
    \left( g _1 \right) \right) &&\eqqcolon \overleftarrow{ X } \left(
    t _0, t _1, g _1 \right) ,\\
  \overrightarrow{ \left( 0, X \right) } \left( t _0, t _1 , g _1
  \right) &= \left( 0 _{ t _0 } , 0 _{ t _1 } , \overrightarrow{ X }
    \left( g _1 \right) \right) &&\eqqcolon \overrightarrow{ X } \left(
    t _0, t _1, g _1 \right) .
\end{alignat*}

Now, suppose we are given a constraint submanifold $ N _{ \mathbb{R} }
\subset G _{ \mathbb{R} } $ and a discrete Lagrangian $ L _{
  \mathbb{R} } \colon N _{ \mathbb{R} } \rightarrow \mathbb{R} $.
Then, following the theory presented in \autoref{sec:constraints}, the
discrete dynamics correspond to the Lagrangian submanifold $ \Sigma _{
  L _{ \mathbb{R} } } $ of the cotangent groupoid $ T ^\ast G _{
  \mathbb{R} } \rightrightarrows A ^\ast G _{ \mathbb{R} } $.  More
explicitly, take an arbitrary extension $ \widehat{L} _{ \mathbb{R} }
$, and suppose that $ N _{ \mathbb{R} } $ is given by the vanishing of
constraint functions $ \phi ^a _{ \mathbb{R} } $, $ a \in A $.  Then
the discrete dynamical equations are
\begin{equation*}
  \phi _{ \mathbb{R}  } ^a \left( t _k, t _{ k + 1 } , g _{ k + 1 }
  \right) = 0 , \quad a \in A ,
\end{equation*}
and
\begin{align*}
  \tfrac{\partial }{ \partial t ^\prime } \bigl[ \widehat{L} _{
    \mathbb{R} } + \left( \lambda _k \right) _a \phi ^a _{ \mathbb{R}
  } \bigr] \left( t _{ k - 1 } , t _k , g _k \right) + \tfrac{\partial
  }{ \partial t } \bigl[ \widehat{L} _{ \mathbb{R} } + \left( \lambda
    _{k+1} \right) _a \phi ^a _{ \mathbb{R} } \bigr] \left( t
    _k , t _{ k + 1 }, g _{k+1} \right) &= 0 , \\
  \overleftarrow{ X } \bigl[ \widehat{L} _{ \mathbb{R} } + \left(
    \lambda _k \right) _a \phi ^a _{ \mathbb{R} } \bigr] \left( t _{ k
      - 1 } , t _k , g _k \right) - \overrightarrow{ X } \bigl[
  \widehat{L} _{ \mathbb{R} } + \left( \lambda _{k+1} \right) _a \phi
  ^a _{ \mathbb{R} } \bigr] \left( t
    _k , t _{ k + 1 }, g _{k+1} \right) &= 0 ,
\end{align*}
for $ k = 1, \ldots, n - 1 $.

\begin{example}
  Let us construct an integrator for the constrained optimal control
  problem on a Lie group, as in \autoref{sec:optimalControl}, but
  incorporating adaptive time-stepping, rather than a fixed time step
  $h$.  Define the time-dependent discrete Lagrangian
  \begin{equation*}
    \widehat{L} _{ \mathbb{R}  } \left( t _{ k - 1 } , t _k , g _k
    \right) = \left( t _k - t _{ k - 1 } \right) \mathfrak{l} \left(
      \frac{ \tau ^{-1} \left( g _k \right) }{ t _k - t _{ k - 1 } } 
    \right) ,
  \end{equation*}
  along with the constraint functions
  \begin{equation*}
    \phi ^a _{ \mathbb{R}  } \left( t _{ k - 1 }, t _k , g _k \right)
    = \left( t _k - t _{ k - 1 } \right) \Psi ^a \left( \frac{ \tau
        ^{-1} \left( g _k \right) }{ t _k - t _{ k - 1 } } \right) ,
    \quad a \in A .
  \end{equation*}
  Taking $ \xi _k = \tau ^{-1} \left( g _k \right) /h _k $, where $ h
  _k = t _k - t _{ k - 1 } $, the discrete dynamics are thus given by
  the equations
  \begin{align*}
    0 &= \Psi ^a \left( \xi _{k+1} \right) , \\
    0 &= \bigl( \mathrm{d} \tau _{ h _k \xi _k } ^{-1} \bigr) ^\ast
    \left( \mathrm{d} \mathfrak{l} \left( \xi _k \right) + \left(
        \lambda _k \right) _a \mathrm{d} \Psi ^a \left( \xi _k \right)
    \right) \\
    &\qquad - \bigl( \mathrm{d} \tau _{ - h _{k+1} \xi _{ k + 1 } } ^{-1}
    \bigr) ^\ast \left( \mathrm{d} \mathfrak{l} \left( \xi _{ k + 1 }
      \right) + \left( \lambda _{ k + 1 } \right) _a \mathrm{d} \Psi
      ^a \left( \xi _{ k + 1 } \right) \right) , \\
    0 &= \mathfrak{l} \left( \xi _k \right) - \left\langle \mathrm{d}
      \mathfrak{l} \left( \xi _k \right) , \xi _k \right\rangle -
    \mathfrak{l} \left( \xi _{ k + 1 } \right) + \left\langle
      \mathrm{d} \mathfrak{l} \left( \xi _{ k + 1 } \right) , \xi _{ k
        + 1 } \right\rangle .
  \end{align*}
  The last condition, corresponding to the $ \frac{ \partial
  }{ \partial t } $ terms, can be interpreted as \emph{conservation of
    energy} along the discrete evolution, as with the
  symplectic-energy-momentum preserving methods of \citet{KaMaOr1999}.
\end{example}

\begin{example}
  The previous example allowed for variable time steps, with
  constraints placed only on the Lie groupoid $G$.  On the other hand,
  we may consider time-dependent discrete Lagrangian mechanics which
  are \emph{unconstrained} on $G$, but with \emph{fixed} time step
  size $ t _{ k } - t _{ k - 1 } = h $ for $ k = 1, \ldots, n $.  As
  before, consider the groupoid $ G _{ \mathbb{R} } \rightrightarrows
  \mathbb{R} \times Q $, and define the constraint submanifold
  \begin{equation*}
    N _{ \mathbb{R}  } = \left\{ \left( t _0, t _1, g _1 \right) \in G
      _{ \mathbb{R}  } \;\middle\vert\; t _1 = t _0 + h \right\} ,
  \end{equation*}
  for some constant $h$.  This corresponds to the vanishing of the
  single constraint function
  \begin{equation*}
    \phi _{ \mathbb{R}  } \left( t _0, t _1, g _1 \right) = t _1 - t _0 - h .
  \end{equation*}
  Then, given an extended discrete Lagrangian $ \widehat{L} _{
    \mathbb{R} } \colon G _{ \mathbb{R} } \rightarrow \mathbb{R} $,
  the equations of motion are (together with the constraint equation)
  \begin{align*}
    \tfrac{\partial }{ \partial t ^\prime } \bigl[ \widehat{L} _{
      \mathbb{R} } + \lambda _k \phi _{ \mathbb{R} } \bigr] \left( t
      _{ k - 1 } , t _k , g _k \right) + \tfrac{\partial }{ \partial t
    } \bigl[ \widehat{L} _{ \mathbb{R} } + \lambda _{k+1} \phi _{
      \mathbb{R} } \bigr] \left( t
      _k , t _{ k + 1 }, g _{k+1} \right) &= 0 , \\
    \overleftarrow{ X } \bigl[ \widehat{L} _{ \mathbb{R} } +
    \lambda _k \phi _{ \mathbb{R} } \bigr] \left( t _{ k - 1 } , t _k
      , g _k \right) - \overrightarrow{ X } \bigl[ \widehat{L} _{
      \mathbb{R} } + \lambda _{k+1} \phi _{ \mathbb{R} } \bigr] \left(
      t _k , t _{ k + 1 }, g _{k+1} \right) &= 0 ,
  \end{align*}
  for $ k = 1, \ldots, n - 1 $.  However, this can be simplified
  greatly, since we observe that $ \frac{ \partial }{ \partial t
    ^\prime } \left[ \phi _{ \mathbb{R} } \right] = 1 $, $
  \frac{ \partial }{ \partial t } \left[ \phi _{ \mathbb{R} } \right]
  = - 1 $, and $ \overleftarrow{ X } \left[ \phi _{ \mathbb{R}
    }\right] = \overrightarrow{ X } \left[ \phi _{ \mathbb{R} }
  \right] = 0 $.  Therefore, the last two equations become
  \begin{align*}
    \tfrac{\partial }{ \partial t ^\prime } \bigl[ \widehat{L} _{
      \mathbb{R} } \bigr] \left( t _{ k - 1 } , t _k , g _k \right) +
    \lambda _k + \tfrac{\partial }{ \partial t } \bigl[ \widehat{L} _{
      \mathbb{R} } \bigr] \left( t _k , t _{ k + 1 }, g _{k+1}
    \right) - \lambda _{ k + 1 }  &= 0 , \\
 \overleftarrow{ X } \bigl[ \widehat{L} _{ \mathbb{R} } \bigr]
    \left( t _{ k - 1 } , t _k , g _k \right) - \overrightarrow{ X }
    \bigl[ \widehat{L} _{ \mathbb{R} } \bigr] \left( t _k , t _{ k + 1
      }, g _{k+1} \right) &= 0 .
  \end{align*}
  Finally, observe that these two equations are completely decoupled,
  so we can in fact eliminate the first equation.  Therefore, we
  obtain 
  \begin{equation*}
    \overleftarrow{ X } \bigl[ \widehat{L} _{ \mathbb{R} } \bigr]
    \left( t _{ k - 1 } , t _k , g _k \right) - \overrightarrow{ X }
    \bigl[ \widehat{L} _{ \mathbb{R} } \bigr] \left( t _k , t _{ k + 1
      }, g _{k+1} \right) = 0 ,
  \end{equation*}
  with $ t _{ k + 1 } = t _k + h $, for $ k = 1, \ldots, n - 1 $.

  This has precisely the same form as the discrete Euler--Lagrange
  equations in the time-independent case; as in the continuous theory
  of Lagrangian mechanics, no ``$ \frac{ \partial L }{ \partial t }
  $''-type terms arise from introducing the time dependency.  In the
  special case where $ \widehat{ L } _{ \mathbb{R} } $ depends only on
  the time step size $ t _{ k + 1 } - t _k $, this is equivalent to
  defining the time-independent discrete Lagrangian $ L \colon G
  \rightarrow \mathbb{R} $, and one recovers precisely the usual,
  unconstrained discrete Euler--Lagrange equations.

  Finally, we remark that an extension of this setup can be used for
  more sophisticated step size control, by taking the constraint
  function to be (for example),
  \begin{equation*}
    \phi _{ \mathbb{R}  } \left( t _0, t _1 , g _1 \right) = t _1 - t
    _0 - h \left( g  _1 \right) ,
  \end{equation*}
  where $ h \colon G \rightarrow \mathbb{R} $ is some step size
  function.  In this case, $ \overleftarrow{ X } \left[ \phi _{
      \mathbb{R} }\right] = - \overleftarrow{ X } \left[ h \right] $
  and $ \overrightarrow{ X } \left[ \phi _{ \mathbb{R} }\right] = -
  \overrightarrow{ X } \left[ h \right] $, neither of which is
  generally zero.  This differs considerably from the constant $h$
  case treated above: in general, the equations of motion no longer
  decouple, and moreover the discrete Euler--Lagrange equations
  contain additional terms arising from the time step control
  function.  This gives some insight into the delicate nature of
  implementing time step control for structure-preserving numerical
  integrators (as discussed at length in, e.g., \citet{HaLuWa2006}).
\end{example}

\section{Conclusion}
\label{sec:conclusion}

\subsection{Summary of results}

We began this paper by developing a generalized theory of discrete
Lagrangian mechanics, in terms of Lagrangian submanifolds of
symplectic groupoids, which induce Poisson relations on their base
manifolds.  We also characterized the regularity and reversibility of
these systems, providing a significant generalization of previous
results.

Applying this framework to the cotangent groupoid $ T ^\ast G
\rightrightarrows A ^\ast G $ of a groupoid $ G \rightrightarrows Q $,
we were able to formulate a new theory of discrete constrained
Lagrangian mechanics, where the discrete Lagrangian $L$ is defined on
a constraint submanifold $ N \subset G $; this reduces to the earlier
unconstrained theory in the special case $ N = G $.  This allows both
for integrable constraints, when $N$ is a subgroupoid of $G$, as well
as more general non-integrable constraints, and the distinction
between the two is consistent with the continuous theory on Lie
algebroids.  The Lagrangian submanifold $ \Sigma _L \subset T ^\ast G
$ generated by $L$ was shown to have the structure of an affine bundle
over $N$, associated to the conormal bundle of $N$.  When $N$ is
defined implicitly, by the vanishing of a family of functions $ \phi
^a $ in a neighborhood on $G$, we showed that one can obtain natural
coordinates for the fibers of $ \Sigma _L $, which correspond to
Lagrange multipliers for the constraints.  The resulting dynamics on $
\Sigma _L $, therefore, specify the evolution of configurations (on
the base) and Lagrange multipliers (on the fibers).

In this setting, we also introduced the notion of a morphism between
discrete constrained Lagrangian systems, which corresponds to a
morphism of Lie groupoids that preserves the additional structure of
the discrete Lagrangians and constraint manifolds.  With this
definition, it was then proven that morphisms of discrete constrained
Lagrangian systems allow for reduction of the dynamics.  In addition,
we studied Noether symmetries, and proved a discrete version of
Noether's theorem, for these constrained systems.

After this, we proved that, under some regularity conditions, the
dynamical equations can also be derived from a variational principle,
where the discrete action sum is constrained to trajectories lying in
$N$.  This establishes a connection between these systems and
\emph{variational integrators}.

Finally, we applied the previous results to discretize several
examples of continuous systems with constraints, including constrained
mechanics and optimal control on Lie groups, the non-integrable
plate-ball system, and time-dependent systems with either fixed or
adaptive time steps.

\subsection{Future directions} There are many interesting directions,
regarding discrete constrained Lagrangian mechanics, that remain to be
explored.  For example, it is well known that if we have a regular,
continuous Lagrangian function $ L \colon T Q \rightarrow \mathbb{R}
$, and a sufficiently small time step $ h > 0 $, then one may define a
regular discrete Lagrangian function $ L _h ^E \colon Q \times Q
\rightarrow \mathbb{R} $, called the \emph{exact discrete Lagrangian},
where $ L _h ^E \left( q _0, q _1 \right) $ is precisely the action
integral along the Euler--Lagrange path from $ q _0 $ to $ q _1 $,
over a time interval of size $h$.  Moreover, the exact Hamiltonian
flow, for time $h$, is just the pushforward (via the discrete Legendre
transformations) of the discrete flow associated with $ L _h ^E $
(cf.~\citet{MaWe2001}).  In other words, the Lagrangian submanifold $
\mathrm{d} L _h ^E \subset T ^\ast \left( Q \times Q \right) $ is
equal to the graph of the Hamiltonian flow for time $h$.  This result
is important for the error analysis of variational integrators on $ Q
\times Q $, and this analysis can also be extended to an exact
discrete Lagrangian $ L _h ^E \colon G \rightarrow \mathbb{R} $ when
the continuous Lagrangian $ L \colon A G \rightarrow \mathbb{R} $ is
defined on the Lie algebroid of $G$ (\citet{MaMaMa2012}).  It would be
interesting to extend this construction for a \emph{constrained}
Lagrangian system, with constraint distribution $ \Delta \subset A G
$, and to see if one could express such results in terms of a
Lagrangian bisection of $ T ^\ast G $, generated by an exact discrete
Lagrangian on some $ N \subset G $.

Finally, for continuous systems, \citet{YoMa2006b} showed that a
variational principle, which they call the \emph{Hamilton--Pontryagin
  principle}, is closely related to Dirac structures, and thus can be
quite useful for studying systems with nonholonomic constraints and
various degeneracies.  In \citet{Stern2010}, it was shown that a
\emph{discrete Hamilton--Pontryagin principle} describes the
relationship between the generating-function and variational
interpretations of the unconstrained discrete Lagrangian $ L \colon G
\rightarrow \mathbb{R} $.  It would be interesting to see if this
variational principle could be generalized to discrete constrained
Lagrangians $ L \colon N \subset G \rightarrow \mathbb{R} $, and how
this might be connected to the discretization of Dirac structures and
of the Courant algebroid.

\begin{acknowledgment}
  We are grateful to Eduardo Mart\'inez for providing helpful feedback
  at various stages of this research.  This work has been partially
  supported by MICINN (Spain) grants MTM2009-13383,
  MTM2010-21186-C02-01, and MTM2009-08166-E; project ACIISI
  SOLSUBC200801000238; project ``IngenioMathematica'' (i-MATH)
  No.~{CSD} 2006-00032 (Consolider-Ingenio 2010); and IRSES project
  ``Geomech-246981.''  Thanks also to the ``Research in Teams''
  program of the Trimester in Combinatorics and Control (CoCo) 2010,
  funded by NSF Award DMS 09-60589, for providing financial support
  and stimulating discussions.  In addition, A.~S. was supported by
  NSF DMS/CM Award 0715146 and PHY/PFC Award 0822283 during a
  postdoctoral fellowship at the University of California, San Diego.
\end{acknowledgment}

\appendix

\section{Bisections of Lie groupoids}
\label{app:bisections}

In this appendix, we will recall some key facts about bisections of
Lie groupoids, which provide the technical underpinning for several
concepts appearing elsewhere in this article.  Most notably, we derive
an expression for the tangent map of Lie groupoid multiplication.
This expression, which appeared in \citet{Xu1995}, is a crucial
ingredient in the proof of the \autoref{thm:variationalIsomorphism},
which we employed in the variational formulation of discrete
constrained Lagrangian mechanics.  We also apply this result to
present an alternative derivation of the multiplication map for the
cotangent groupoid.  (Our definitions largely follow those of
\citet[Chapter 15]{CaWe1999}.)

\subsection{Global and local bisections}
Let $ G \rightrightarrows Q $ be a Lie groupoid with source map $
\alpha \colon G \rightarrow Q $ and target map $ \beta \colon G
\rightarrow Q $.  A submanifold $ \mathcal{B} \subset G $ is called a
\emph{bisection} of $G$ if the restricted maps, $ \alpha \rvert _{
  \mathcal{B} } \colon \mathcal{B} \rightarrow Q $ and $ \beta \rvert
_{ \mathcal{B} } \colon \mathcal{B} \rightarrow Q $ are both
diffeomorphisms.  Consequently, for any bisection $\mathcal{B} \subset
G $, there is a corresponding $\alpha$-section $ \mathcal{B} _\alpha =
\left( \alpha \rvert _{ \mathcal{B} } \right) ^{-1} \colon Q
\rightarrow G $, where $ \beta \circ \mathcal{B} _{ \alpha } \colon Q
\rightarrow Q $ is a diffeomorphism.  (\citet{Mackenzie2005} defines a
bisection to be the map $ \mathcal{B} _\alpha $ having this property,
rather than its image $\mathcal{B}$; the definitions are equivalent.)
Likewise, there is a $\beta$-section $ \mathcal{B} _\beta = \left(
  \beta \rvert _{ \mathcal{B} } \right) ^{-1} \colon Q \rightarrow G
$, where $ \alpha \circ \mathcal{B} _\beta = \left( \beta \circ
  \mathcal{B} _\alpha \right) ^{-1} \colon Q \rightarrow Q $ is a
diffeomorphism.  Furthermore, each bisection $ \mathcal{B} \subset G $
defines a \emph{left action} $ \ell _{ \mathcal{B} } \colon G
\rightarrow G $ and a \emph{right action} $ r _{ \mathcal{B} } \colon
G \rightarrow G $ on the groupoid, given on any $ g \in G $ by
\begin{equation*}
  \ell _{ \mathcal{B} } (g) = \mathcal{B} _\beta \left( \alpha (g)
  \right) g, \qquad r _{ \mathcal{B} } (g) = g \mathcal{B} _\alpha
  \left( \beta (g) \right) .
\end{equation*}
Given two bisections $ \mathcal{B} _1 , \mathcal{B} _2 \subset G $,
one can show that the product $ \mathcal{B} _1 \mathcal{B} _2 = \ell
_{ \mathcal{B} _1 } \left( \mathcal{B} _2 \right) = r _{ \mathcal{B}
  _2 } \left( \mathcal{B} _1 \right) $ is again a bisection, and that
the bisections of $G$ in fact form a group.

More generally, $ B \subset G $ is called a \emph{local bisection} if
the restricted maps $ \alpha \rvert _B $ and $ \beta \rvert _B $ are
local diffeomorphisms onto open sets $ U , V \subset Q $,
respectively.  Analogously to the global case, there exists a local
$\alpha$-section $ B _\alpha \colon U \rightarrow B $ such that $
\beta \circ B _\alpha \colon U \rightarrow V $ is a diffeomorphism, as
well as a local $\beta$-section $ B _\beta \colon V \rightarrow B $
such that $ \alpha \circ B _\beta = \left( \beta \circ B _\alpha
\right) ^{-1} \colon V \rightarrow U $ is a diffeomorphism.  (As
before, an alternate but equivalent definition,
cf.~\citet{Mackenzie2005}, takes a local bisection to be the map $ B
_\alpha $, rather than $ B $ itself.)  Each local bisection $ B
\subset G $ defines a \emph{local left action} $ \ell _B \colon \alpha
^{-1} (V) \rightarrow \alpha ^{-1} (U) $ and a \emph{local right
  action} $ r _B \colon \beta ^{-1} (U) \rightarrow \beta ^{-1} (V) $,
given by
\begin{equation*}
  \ell _B = B _\beta \left( \alpha (g) \right) g , \qquad r _B (g) = g
  B _\alpha \left( \beta (g) \right) .
\end{equation*}
It is also possible to define multiplication of local bisections;
moreover, if $ i \colon G \rightarrow G $ denotes the inversion map on
$G$, then the inverse $ i (B) $ of a local bisection is again a local
bisection.  Thus, the local bisections of $G$ form an \emph{inverse
  semigroup}, which contains the group of global bisections as those
elements with $ U = V = Q $.

\subsection{The tangent map of Lie groupoid multiplication}

A key fact about Lie groupoids is that, for every $ g \in G $, there
exists a local bisection $ B \subset G $ such that $ g \in B $. (In
general, though, there may not exist any \emph{global} bisection
through $g$.)  Using this property, we may now state and prove the
following theorem on the tangent map of multiplication in $G$.

\begin{theorem}[\citet{Xu1995}]
  \label{thm:tangentMultiplication}
  Let $ G \rightrightarrows Q $ be a Lie groupoid, with multiplication
  map $ m \colon G _2 \rightarrow G $.  Suppose $ \left( g _1 , g _2
  \right) \in G _2 $ is a pair of composable elements, and denote $ q
  = \beta \left( g _1 \right) = \alpha \left( g _2 \right) $.  Then
  the tangent map $ T _{ \left( g _1 , g _2 \right) } m \colon T _{
    \left( g _1, g _2 \right) } G _2 \rightarrow T _{ g _1 g _2 } G $
  is given by
  \begin{equation*}
    T _{ \left( g _1 , g _2 \right) } m \left( v _{ g _1 }, v _{ g _2
      } \right) = T _{ g _1 } r _{ B _2 } \left( v _{ g _1 } \right) +
    T _{ g _2 } \ell _{ B _1 } \left( v _{ g _2 } \right) - T _q
    \left( \ell _{ B _1 } \circ r _{ B _2 } \circ \epsilon \right)
    \left( v _q \right) ,
  \end{equation*} 
  where $ B _1 $ and $ B _2 $ are local bisections such that $ g _1
  \in B _1 $, $ g _2 \in B _2 $, and where $ v _q = T _{ g _1 } \beta
  \left( v _{ g _1 } \right) = T _{ g _2 } \alpha \left( v _{ g _2 }
  \right) \in T _q Q $.
\end{theorem}

\begin{proof}
  Denote by $ \left( B _2 \right) _\alpha $ and $ \left( B _1 \right)
  _\beta $, respectively, the local $\alpha$- and $\beta$-sections
  corresponding to $ B _2 $ and $ B _1 $.  This implies that $ g _2 =
  \left( B _2 \right) _\alpha (q) $ and $ g _1 = \left( B _1 \right)
  _\beta (q) $, so we decompose the tangent vector $ \left( v _{ g _1 }, v _{ g _2
    } \right) \in T _{ \left( g _1, g _2 \right) } G _2 $ as
  \begin{multline}
    \label{eqn:tmdecomp}
    \left( v _{ g _1 }, v _{ g _2 } \right) = \left( v _{ g _1 } , T
      _q \left( B _2 \right) _\alpha \left( v _q \right) \right) + ( T
    _q \left( B _1 \right)_ \beta \left( v _q \right) , v _{ g _2
    } ) \\
    - ( T _q \left( B _1 \right)_ \beta \left( v _q \right), T _q
    \left( B _2 \right) _\alpha \left( v _q \right) ) .
  \end{multline} 
  Applying the tangent map to each of these terms individually, we
  observe for the first term that
\begin{align*}
  T _{\left( g _1 , g _2 \right) }m \left( v _{ g _1 } , T _q \left( B
      _2 \right) _\alpha \left( v _q \right) \right) &= T _{ \left( g
      _1 , g _2 \right) } m \left( T _{ g _1 } \mathrm{id} _G \left( v
      _{ g _1 } \right) , T _{ g _1 } \left( \left( B _2 \right)
      _\alpha \circ \beta \right) \left( v _{ g _1 } \right)
  \right) \\
  &= T _{g _1} \left( m \circ \left( \mathrm{id} _G , \left( B _2
      \right) _\alpha \circ \beta \right) \right) \left( v _{ g _1 }
  \right).
\end{align*} 
Now, since
\begin{equation*}
  \left( m \circ \left( \mathrm{id} _G , \left( B _2 \right) _\alpha
      \circ \beta \right) \right) (g) = m \left( g, \left( B _2 \right)
    _\alpha \beta (g) \right) = g \left( B _2 \right) _\alpha \beta
  (g) = r _{ B _2 } (g) ,
\end{equation*} 
we conclude that 
\begin{equation}\label{eqn:tm1}
  T _{\left( g _1 , g _2 \right) }m \left( v _{ g _1 } , T _q \left( B
      _2 \right) _\alpha \left( v _q \right) \right) = T _{ g _1 } r
  _{ B _2 } \left( v _{ g _1 } \right) .
\end{equation} 
Similarly, applying the tangent map to the second term of
\eqref{eqn:tmdecomp}, it follows that
\begin{equation}\label{eqn:tm2}
  T _{\left( g _1 , g _2 \right) }m ( T
    _q \left( B _1 \right)_ \beta \left( v _q \right) , v _{ g _2
    } ) = T _{ g _2 } \ell _{ B _1 } \left( v _{ g _2 } \right) .
\end{equation}
On the other hand, for the third term of \eqref{eqn:tmdecomp}, we have
\begin{equation*}
  T _{\left( g _1 , g _2 \right) }m ( T _q \left( B _1 \right)_ \beta
  \left( v _q \right), T _q \left( B _2 \right) _\alpha \left( v _q
  \right) ) =   T _q \left( m \circ ( ( B _1 ) _\beta , (B _2 )
    _\alpha ) \right)  (v _q ) ,
\end{equation*}
and since
\begin{equation*}
  \left( m \circ ( ( B _1 ) _\beta , (B _2 ) _\alpha ) \right) (q) = 
  \left( B _1 \right) _\beta (q) \left( B _2 \right) _\alpha
  (q) = \left( \ell _{ B _1 } \circ r _{ B _2 } \circ \epsilon \right)
  (q) ,
\end{equation*} 
we deduce that
\begin{equation}\label{eqn:tm3}
  T _{\left( g _1 , g _2 \right) }m ( T _q \left( B _1 \right)_ \beta
  \left( v _q \right), T _q \left( B _2 \right) _\alpha \left( v _q
  \right) ) = T _q \left( \ell _{ B _1 } \circ r _{ B _2 } \circ
    \epsilon \right) \left( v _q \right) .
\end{equation} 
Finally, substituting the expressions \eqref{eqn:tm1},
\eqref{eqn:tm2}, \eqref{eqn:tm3} for the respective terms of
\eqref{eqn:tmdecomp} yields the result.
\end{proof}

\subsection{Multiplication in the cotangent groupoid}

Given a Lie groupoid $ G \rightrightarrows Q $ with Lie algebroid $ A
G \rightarrow Q $, we have already discussed the source and target
maps, respectively denoted $ \widetilde{ \alpha } $ and $ \widetilde{
  \beta } $, of the cotangent groupoid $ T ^\ast G \rightrightarrows A
^\ast G $.  In this subsection, we discuss the multiplication map $
\widetilde{ m } $, whose formulation depends on the properties of
bisections.  In particular, we derive an expression for this
multiplication map using \autoref{thm:tangentMultiplication}.

The multiplication map $ \widetilde{ m } $ on $ T ^\ast G $ can be
characterized by the following two conditions
(cf.~\citet{Mackenzie2005}):
\begin{enumerate}
\itemsep\topsep
\item if $ \left( \mu _{ g _1 }, \mu _{g _2 } \right) \in T ^\ast _{ g
    _1 } G \times T ^\ast _{ g _2 } G $ is any composable pair, i.e.,
  $ \widetilde{ \beta } \left( \mu _{ g _1 } \right) = \widetilde{
    \alpha } \left( \mu _{ g _2 } \right) $, then the product $
  \widetilde{ m } \left( \mu _{ g _1 }, \mu _{ g _2 } \right) $ lies
  in $ T ^\ast _{ g _1 g _2 } G $;

\item if $ \left( v _{ g _1 }, v _{ g _2 } \right) \in T _{ \left( g
      _1, g _2 \right) } G _2 $, then $
  \left\langle \widetilde{ m } \left( \mu _{ g _1 }, \mu _{ g _2 }
    \right) , T _{ \left( g _1, g _2 \right) } m \left( v _{ g _1 }, v
      _{ g _2 } \right) \right\rangle = \left\langle \mu _{ g _1 }, v
    _{ g _1 } \right\rangle + \left\langle \mu _{ g _2 }, v _{ g _2 }
  \right\rangle $. \label{condition:cotangentII}
\end{enumerate}
Using these properties, we can now obtain the following explicit
formula for the cotangent multiplication map.

\begin{theorem}
  Let $ \left( \mu _{ g _1 } , \mu _{ g _2 } \right) \in T ^\ast _{ g
    _1 } G \times T ^\ast _{ g _2 } G $ be a composable pair, with $
  \beta \left( g _1 \right) = \alpha \left( g _2 \right) = q \in Q $
  and $ \widetilde{ \beta } \left( \mu _{ g _1 } \right) = \widetilde{
    \alpha } \left( \mu _{ g _2 } \right) = \mu _q \in A ^\ast _q G $.
  Define the linear epimorphism
  \begin{equation*}
    \pi _q \colon T _{ \epsilon (q) } G \rightarrow A _q G , \quad  v
    _{ \epsilon (q) } \mapsto v _{ \epsilon (q) } - T _{ \epsilon (q)
    } \left( \epsilon \circ \alpha \right) \left( v _{ \epsilon (q) }
    \right) .
  \end{equation*} 
  Then the cotangent multiplication map $ \widetilde{ m } $ is given by
    \begin{multline*}
      \widetilde{ m } \left( \mu _{ g _1 } ,\mu _{ g _2 } \right) \\
      = T ^\ast _{ g _1 g _2 } \ell _{ i \left( B _1 \right) } \left(
        \mu _{ g _2 } \right) + T ^\ast _{ g _1 g _2 } r _{ i \left( B
          _2 \right) } \left( \mu _{ g _1 } \right) - \left( T ^\ast
        _{ g _1 g _2 } \left( \ell _{ i \left( B _1 \right) } \circ r
          _{ i \left( B _2 \right) } \right) \circ \pi _q ^\ast
      \right) \left( \mu _q \right),
    \end{multline*} 
    where $ B _1 $ and $ B _2 $ are local bisections of $G$ such that
    $ g _1 \in B _1 $, $ g _2 \in B _2 $.
\end{theorem}

\begin{proof}
  First, we calculate
  \begin{align*}
    \left\langle \mu _{ g _2 } , v _{ g _2 } \right\rangle &=
    \left\langle \mu _{ g _2 } , T _{ g _2 } \left( \ell _{ i \left( B
            _1 \right) } \circ \ell _{ B _1 } \right) \left( v _{ g _2
        } \right) \right\rangle \\
    &= \left\langle T ^\ast _{ g _1 g _2 } \ell _{ i \left( B _1
        \right) } \left( \mu _{ g _2 } \right) , T _{ g _2 } \ell _{ B
        _1 } \left(
        v _{ g _2 } \right) \right\rangle \\
    &=\!
    \begin{multlined}[t]
      \langle T ^\ast _{ g _1 g _2 } \ell _{ i \left( B _1 \right) }
      \left( \mu _{ g _2 } \right) , T _{ \left( g _1, g _2 \right) }
      m \left( v _{ g _1 } , v _{ g _2 } \right) - T
      _{ g _1 } r _{ B _2 } \left( v _{ g _1 } \right) \\
      + T _q \left( \ell _{ B _1 } \circ r _{ B _2 } \circ \epsilon
      \right) \left( v _q \right) \rangle ,
    \end{multlined}
  \end{align*}
  where the last equality is obtained by applying
  \autoref{thm:tangentMultiplication}.  Observe that
  \begin{multline*}
    \left\langle T ^\ast _{ g _1 g _2 } \ell _{ i \left( B _1 \right)
      } \left( \mu _{ g _2 } \right), T _{ g _1 } r _{ B _2 } \left( v
        _{ g _1 } \right) - T _q \left( \ell _{ B _1 } \circ r _{ B _2
        } \circ \epsilon \right) \left( v _q \right) \right\rangle \\
    \begin{aligned}
      &= \left\langle \mu _{ g _2 } , T _{ g _1 } \left( \ell _{ i
            \left( B _1 \right) } \circ r _{ B _2 } \right) \left( v
          _{ g _1 } \right) - T _q \left( r _{ B _2 } \circ \epsilon
        \right) \left( v _q \right) \right\rangle \\
      &= \left\langle \mu _{ g _2 } , T _{ \epsilon (q) } r _{ B _2 }
        \left( T _{ g _1 } \ell _{ i \left( B _1 \right) } \left( v _{
              g _1 } \right) - T _{ \epsilon (q) } \left( \epsilon
            \circ \beta \right) \left( T _{ g _1 } \ell _{ i \left( B
                _1 \right) } \left( v _{ g _1 } \right) \right)
        \right) \right\rangle \\
      &= \left\langle \mu _{ g _2 } , - T _{ \epsilon (q) } \left( r
          _{ B _2 } \circ i \right) \left( \pi _q \left( T _{ g _1 }
            \ell _{ i \left( B _1 \right) } \left( v _{ g _1 } \right)
          \right) \right) \right\rangle \\
      &= \left\langle \pi _q ^\ast \widetilde{ \alpha } \left( \mu _{
            g _2 } \right) , T _{ g _1 } \ell _{ i \left( B _1 \right)
        } \left( v _{ g _1 } \right) \right\rangle ,
    \end{aligned}
  \end{multline*}
  so altogether, we have
  \begin{multline*}
    \left\langle \mu _{ g _2 } , v _{ g _2 } \right\rangle =
    \left\langle T ^\ast _{ g _1 g _2 } \ell _{ i \left( B _1 \right)
      } \left( \mu _{ g _2 } \right) , T _{ \left( g _1, g _2 \right)
      } m \left( v _{ g _1 } , v _{ g _2 } \right)
    \right\rangle \\
    - \left\langle \pi _q ^\ast \left( \mu _q \right) , T _{ g _1 }
      \ell _{ i \left( B _1 \right)} \left( v _{ g _1 } \right)
    \right\rangle .
  \end{multline*} 
Following essentially the same procedure, one also obtains
\begin{multline*}
  \left\langle \mu _{ g _1 } , v _{ g _1 } \right\rangle =
  \left\langle T ^\ast _{ g _1 g _2 } r _{ i \left( B _2 \right) }
    \left( \mu _{ g _1 } \right) , T _{ \left( g _1, g _2 \right) }
    m \left( v _{ g _1 } , v _{ g _2 } \right)
  \right\rangle \\
  - \left\langle \pi _q ^\ast \left( \mu _q \right) , T _{ g _2 } r
    _{ i \left( B _2 \right)} \left( v _{ g _2 } \right)
  \right\rangle ,
\end{multline*} 
so adding these together,
\begin{multline*}
  \left\langle \mu _{ g _1 } , v _{ g _1 } \right\rangle +
  \left\langle \mu _{ g _2 } , v _{ g _2 } \right\rangle \\
  = \left\langle T ^\ast _{ g _1 g _2 } \ell _{ i \left( B _1 \right)
    } \left( \mu _{ g _2 } \right) + T ^\ast _{ g _1 g _2 } r _{ i
      \left( B _2 \right) } \left( \mu _{ g _1 } \right) , T _{ \left(
        g _1, g _2 \right) } m \left( v _{ g _1 } , v _{ g _2 }
    \right) \right\rangle \\
  - \left\langle \pi _q ^\ast \left( \mu _q \right) , T _{ g _1 } \ell
    _{ i \left( B _1 \right) } \left( v _{ g _1 } \right) + T _{ g _2
    } r _{ i \left( B _2 \right)} \left( v _{ g _2 } \right)
  \right\rangle .
\end{multline*}
However, for this last term, we can use
\autoref{thm:tangentMultiplication} again to write
\begin{multline*}
  T _{ g _1 } \ell _{ i \left( B _1 \right) } \left( v _{ g _1 }
  \right) + T _{ g _2 } r _{ i \left( B _2 \right)} \left( v _{ g _2 }
  \right) \\
  \begin{aligned}
    &= T _{ g _1 g _2 } \left( \ell _{ i \left( B _1 \right) } \circ r
      _{ i \left( B _2 \right) } \right) \left( T _{ g _1 } r _{ B _2
      } \left( v _1 \right) + T _{ g _2 }
      \ell _{ B _1 } \left( v _2 \right) \right) \\
    &= T _{ g _1 g _2 } \left( \ell _{ i \left( B _1 \right) } \circ r
      _{ i \left( B _2 \right) } \right) \left( 
      T _{ \left( g _1, g _2 \right) } m \left( v _{ g _1 } , v _{ g
          _2 } \right) + T _q \left( \ell _{ B _1 } \circ r _{ B _2 }
        \circ \epsilon \right) \left( v _q \right) \right) \\
    &= T _{ g _1 g _2 } \left( \ell _{ i \left( B _1 \right) } \circ r
      _{ i \left( B _2 \right) } \right) \left( 
      T _{ \left( g _1, g _2 \right) } m \left( v _{ g _1 } , v _{ g
          _2 } \right) \right) + T _q \epsilon  \left( v _q \right) .
  \end{aligned}
\end{multline*}
But since 
\begin{equation*}
  \left( \pi _q \circ T _q \epsilon \right) \left( v _q
  \right) = T _q \epsilon \left( v _q \right) - T _q\left( \epsilon \circ
    \alpha \circ \epsilon \right) \left( v _q \right) = T _q \epsilon
  \left( v _q \right) - T _q \epsilon \left( v _q \right) = 0 ,
\end{equation*} 
the $ T _q \epsilon $ term vanishes.  Finally, we are left with
\begin{multline*}
  \left\langle \mu _{ g _1 } , v _{ g _1 } \right\rangle +
  \left\langle \mu _{ g _2 } , v _{ g _2 } \right\rangle 
  = \bigl\langle T ^\ast _{ g _1 g _2 } \ell _{ i \left( B _1 \right)
    } \left( \mu _{ g _2 } \right) + T ^\ast _{ g _1 g _2 } r _{ i
      \left( B _2 \right) } \left( \mu _{ g _1 } \right) \\ 
  - \left( T ^\ast _{ g _1 g _2 } \left( \ell _{ i \left(
            B _1 \right) } \circ r _{ i \left( B _2 \right) } \right)
      \circ \pi _q ^\ast \right) \left( \mu _q \right) , T _{ \left(
        g _1, g _2 \right) } m \left( v _{ g _1 } , v _{ g _2 }
    \right) \bigr\rangle ,
\end{multline*} 
and since this equals $ \left\langle \widetilde{ m } \left( \mu _{ g
      _1 }, \mu _{ g _2 } \right) , T _{ \left( g _1, g _2 \right) } m
  \left( v _{ g _1 } , v _{ g _2 } \right) \right\rangle $, by
condition \hyperref[condition:cotangentII]{(ii)}, above, this
completes the proof.
\end{proof}

\end{document}